\documentclass[11pt]{amsart}

\usepackage{amssymb,amsfonts,latexsym}
\usepackage[centertags]{amsmath}
\usepackage{amsfonts}
\usepackage{amssymb}
\usepackage{amsthm}
\usepackage[all]{xy}
\usepackage{cite}
\usepackage{graphicx,color}
\usepackage{todo}

\newtheorem{cont}{cont}[section]
\newtheorem{theorem}[cont]{Theorem}
\newtheorem{proposition}[cont]{Proposition}
\newtheorem{lemma}[cont]{Lemma}
\newtheorem{corollary}[cont]{Corollary}
\newtheorem{defn}[cont]{Definition}
\newtheorem{pblm}[cont]{Problem}

\newtheorem*{Enunciato*}{Enunciato}
\newtheorem{conj}[cont]{Conjecture}

\numberwithin{equation}{section}
\newtheorem{remark}[cont]{Remark}
\newtheorem{example}[cont]{Example}
\newtheorem*{not*}{Notation}

\newcommand{\cO}{{\mathcal O}}
\newcommand{\shF}{\mathcal{F}}

\newcommand{\shL}{\mathcal{L}}
\newcommand{\shH}{\mathcal{H}}
\newcommand{\cA}{\mathcal{A}}
\newcommand{\cB}{\mathcal{B}}
\newcommand{\cC}{\mathcal{C}}
\newcommand{\cD}{\mathcal{D}}
\newcommand{\shI}{\mathcal{I}}

\newcommand{\cL}{\mathcal{L}}

\newcommand{\shE}{\mathcal{E}}
\newcommand{\PP}{\mathbb{P}}
\newcommand{\ZZ}{\mathbb{Z}}

\newcommand{\odi}[1]{\mathcal{O}_{#1}}

\newcommand{\arr}{\longrightarrow}

\DeclareMathOperator{\Hl}{H} \DeclareMathOperator{\h}{h}
\DeclareMathOperator{\Pic}{Pic} 
\DeclareMathOperator{\rk}{rk} \DeclareMathOperator{\Hom}{Hom}

\DeclareMathOperator{\Spec}{Spec}
\DeclareMathOperator{\depth}{depth}

\DeclareMathOperator{\Proj}{Proj} \DeclareMathOperator{\di}{dim}
\DeclareMathOperator{\codim}{codim}
\DeclareMathOperator{\coker}{coker}
\DeclareMathOperator{\Ext}{Ext} \DeclareMathOperator{\pd}{pd}
\DeclareMathOperator{\Tor}{Tor} 
 \DeclareMathOperator{\im}{im}

\begin{document}
\title{On the normal  sheaf of  determinantal varieties}

\author[Jan O.\ Kleppe, Rosa M.\ Mir\'o-Roig]{Jan O.\ kleppe, Rosa M.\
Mir\'o-Roig$^{*}$}
\address{Oslo and Akershus University College,
         Faculty of Technology, Art and Design,
         PB. 4, St. Olavs Plass, N-0130 Oslo,
         Norway}
\email{JanOddvar.Kleppe@hioa.no}
\address{Facultat de Matem\`atiques,
Departament d'Algebra i Geometria, Gran Via de les Corts Catalanes
585, 08007 Barcelona, SPAIN } \email{miro@ub.edu}

\date{\today}
\thanks{$^*$ Partially supported by MTMM2010-15256.}

\begin{abstract} Let $X$ be a  standard determinantal scheme $X\subset \PP^n$ of codimension $c$, i.e. a scheme defined by the maximal minors of a $t\times (t+c-1)$  homogeneous polynomial matrix $\cA$.  In this paper, we study the main features of its normal sheaf ${\mathcal N}_X$. We prove that under some mild restrictions:  (1) there exists a line bundle $\cL $ on $X\setminus Sing(X)$ such that ${\mathcal N}_X\otimes \cL$ is arithmetically Cohen-Macaulay and, even more, it is Ulrich whenever the entries of   $\cA$ are  linear forms (2) ${\mathcal N}_X$ is simple (hence, indecomposable) and, finally, (3) ${\mathcal N}_X$ is $\mu$-(semi)stable provided the entries of $\cA$ are linear forms.
\end{abstract}

\maketitle

\tableofcontents

\section{Introduction}
The goal of this paper is to study the main properties of the normal  sheaf of a determinantal scheme. Determinantal schemes have been a central topic in both Commutative Algebra and Algebraic Geometry. They are schemes defined by the vanishing of the $r \times r$ minors
of a homogeneous polynomial matrix; and as a classical examples of determinantal schemes we have rational normal scrolls, Segre varieties and Veronese varieties.
On the other hand, the normal sheaf ${\mathcal N}_X$ of a projective scheme $X\subset \PP^n$ has been intensively studied since it reflects many properties of the embedding.
For instance, if $X\subset \PP^n$ is a smooth projective variety of dimension $d\ge n/2$, $n\ge 5$
and  ${\mathcal N}_X$ splits into a sum of line bundles then $X$ is a complete intersection (cf. \cite{B}; Corollary 3.6).
This paper is entirely devoted to study the main features of the normal sheaf ${\mathcal N}_X$ of a (linear) standard determinantal scheme $X\subset \PP^n$, i.e. schemes defined by the maximal minors of a homogeneous polynomial matrix. In particular, we will study the Macaulayness, simplicity and $\mu$-stability of ${\mathcal N}_X$.

\vskip 2mm
Since the  seminal work by Horrocks characterizing ACM  sheaves on $\PP^n$ as  those that completely split into a sum of line bundles (cf. \cite{Hor}), a big amount of papers  devoted to  study ACM sheaves on  projective schemes has appeared. Recall that a sheaf $\shE$ on a projective scheme $X\subset \PP^n$ is arithmetically Cohen-Macaulay (ACM, for short) if it is locally Cohen-Macaulay and $\Hl ^{i}_*(\shE)=\sum_{t\in \ZZ}\Hl ^{i}(X,\shE(t))=0$  for $0<i<\di X$.
 It turns out that a natural way to measure the complexity of a projective scheme is to ask for the families of ACM sheaves that it supports.
 Mimicking an analogous trichotomy in Representation Theory, in \cite{DG} it was proposed a classification of ACM projective varieties as \emph{finite, tame or wild}  according to the complexity of their associated category of ACM vector bundles  and it was proved that this trichotomy is exhaustive for the case of ACM curves: rational curves are finite, elliptic curves are tame and curves of higher genus are wild.
Unfortunately very little is known for varieties of higher dimension.

\vskip 2mm Among ACM vector bundles $\shE $ on a given variety $X$, it is interesting to spot a very important subclass for which its associated module $\oplus_{t}\Hl^0(X,\shE(t))$ has the maximal number of generators, which turns out to be $\deg(X)\rk(\shE)$.
The algebraic counterpart has also arisen a lot of interest. In fact, Ulrich proved
that for a local ring $R$ the minimal number $m(M)$ of generators of a Maximal
Cohen-Macaulay (MCM) module $M$ is bounded by $m (M)\le e(R)\rk(M)$ where
$e(R)$ denotes de multiplicity of $R$ (cf. \cite{Ulr}). MCM modules attaining
this bound are called Ulrich modules. These algebraic considerations prompted
to define, for a projective scheme $X\subset \PP^n$, a sheaf $\shE$ on $X$ to
be Ulrich if it is ACM and the associated graded module $\Hl ^0_*(\shE)$ is
Ulrich. When $\shE$ is initialized (i.e. $\Hl^0(\shE)\ne 0$ but
$\Hl^0(\shE(-1))=0$), the last condition is equivalent to $\di
\Hl^0(X,\shE)=\deg(X)\rk(\shE)$.
The search of Ulrich sheaves on a particular variety is a challenging problem. In fact, few examples of varieties supporting Ulrich sheaves are known, although in \cite{ESW} it has been conjectured that any variety has an Ulrich sheaf. See
\cite{BHU}; Proposition 2.8, for the existence of Ulrich bundles of rank one on linear standard determinantal schemes.
 Moreover, the recent interest in the existence of Ulrich sheaves relies among other things on the fact  that a $d$-dimensional variety $X\subset \PP^n$ supports an Ulrich sheaf (bundle)  if and only if the cone of cohomology tables of coherent sheaves (resp. vector bundles) on $X$ coincides with the cone of cohomology tables of coherent sheaves (resp. vector bundles) on  $\PP^d$ (\cite{ES}; Theorem 4.2).

\vskip 2mm
In this paper, we will address the following longstanding problems:

\begin{pblm} \label{pblm} Let $X\subset \PP^n$ be a (linear) standard determinantal scheme of codimension $c$ and let ${\mathcal
N}_X$ be its normal sheaf:
\begin{itemize} \item[(1)] Is there any invertible sheaf $\shL$ on $X\setminus
Sing(X)$ such that ${\mathcal
N}_X\otimes \shL$ is ACM (resp. Ulrich)? If so, how does the minimal free
$\odi{\PP^n}$-resolution of ${\mathcal
N}_X\otimes \shL$ looks like?
\item[(2)]    Is there any invertible sheaf $\shL'$ on $X\setminus Sing(X)$ such
 that $\wedge ^q{\mathcal
N}_X\otimes \shL'$ is ACM (resp. Ulrich)?
\item[(3)] Is ${\mathcal
N}_X$ simple? or, at least indecomposable?
\item[(4)]Is ${\mathcal N}_X$ $\mu$-(semi)stable?
\end{itemize}
\end{pblm}

In our approach, we often use the well known fact that a locally free sheaf on
the smooth locus of a reduced normal Cohen-Macaulay variety $X$ extends
uniquely to a reflexive sheaf on $X$, or more precisely that the morphism
\eqref{NM} of this paper is an isomorphism under appropriate depth conditions.
We will start our work analyzing whether the normal sheaf ${\mathcal N}_X$
(and its exterior powers) to a standard determinantal scheme is ACM. It is
well known that the normal sheaf ${\mathcal N}_X$ to a standard determinantal
scheme $X\subset \PP^n$ of codimension $c$ is ACM if $1\le c\le 2$ but it is
no longer true for $c\ge 3$. Our first goal will be to prove that under some
weak restriction there exists an invertible sheaf $\shL$ on $X\setminus
Sing(X)$ (i.e. a coherent $\odi{X}$-module that is invertible on $X\setminus
Sing(X)$) such that ${\mathcal N}_X\otimes \shL$ is ACM. Our next aim will be
to prove that for linear standard determinantal schemes $X\subset \PP^n$
satisfying some mild hypothesis, ${\mathcal N}_X\otimes \shL$ is not only ACM
but also Ulrich.

\vskip 2mm As Hartshorne and Casanellas pointed out in \cite{CH} there are few examples of indecomposable ACM bundles of arbitrarily high rank. So, once we know the Macaulayness of a suitable twist  ${\mathcal
N}_X\otimes \shL$ of the normal sheaf  ${\mathcal
N}_X$ to a standard determinantal scheme $X\subset \PP^n$ we are led to ask if it is indecomposable. Concerning Problem \ref{pblm} (3),  we are able to prove that under some weak conditions ${\mathcal
N}_X$ is indecomposable  and, even more, it is simple. As it is explained in section 4, this last  property works in a much more general set up (cf. Theorem \ref{simple}).
Another challenging problem is the existence of  $\mu$-(semi)stable  ACM bundles  of high rank on projective schemes since $\mu$-(semi)stable ACM bundles of higher ranks are essentially unknown due to the lack
of   criteria to check it when the rank is high. In the last part of the paper, we analyze the $\mu$-(semi)stability of the normal sheaf ${\mathcal
N}_X$ to a standard determinantal scheme $X$. We prove that  it is $\mu$-semistable provided  the homogeneous matrix associated to $X$ has linear entries and $\mu$-stable if, in addition, we assume that $X$ has codimension 2.

\vskip 2mm Next we outline the structure of the paper. Section 2 provides the background and basic results on (linear) standard determinantal schemes and the associated complexes needed in the sequel as well as the definition and main properties on ACM sheaves and Ulrich sheaves on a projective scheme. We refer to \cite{b-v} and \cite{eise} for more details on standard determinantal schemes and to \cite{CMP}, \cite{ESW}, \cite{MP2012} and \cite{MR} for more details on Ulrich sheaves. In section 3, we address Problems \ref{pblm} (1) and (2) and we prove our main results. Given a standard determinantal scheme $X\subset \PP^n$ of codimension $c$
associated to a homogeneous $t \times (t+c-1)$ polynomial matrix $\cA$, we denote by $M$ the cokernel of the graded morphism defined by $\cA$ and  we prove that, under some mild assumption,
${\mathcal N}_X\otimes \widetilde{S_{c-2}M}$ is ACM (cf. Theorem \ref{normalresolution}). As a by-product we obtain an  $\odi{\PP^n}$-resolution of ${\mathcal
N}_X\otimes \shL$  which turns out to be pure and linear in the case of linear standard determinantal schemes (cf. Theorem \ref{normalresolution} and Corollary \ref{res}). To prove it we use a generalization of the mapping cone process together with a careful analysis of possible cancelations of repeated summands in the mapping cone construction. Parts of the results of this section are inspired by ideas developed in \cite{K2011} leading to isomorphisms
$\Ext^1(\widetilde{M},\widetilde{S_{i}M})\cong \Hom ({\mathcal I}_X/{\mathcal I}_X^2,\widetilde{S_{i-1}M})$ and to good estimates  of the depth of $\Ext^1(M,S_{i}M)$. We also generalize these results to higher $\Ext $-groups and  to exterior powers of twisted normal sheaves (cf. Theorem \ref{newthm}).
 At the end of this section, we guess a
result, analogous to  Theorem \ref{newthm} for the exterior power of the normal sheaf (cf. Conjectures \ref{conjetura} and \ref{conjetura2}) based on our computations with Macaulay2.
Section 4 deals with Problem \ref{pblm} (3) and we determine conditions under which  the normal sheaf    ${\mathcal N}_X$ of a standard determinantal scheme $X\subset \PP^n$ is simple and, hence, indecomposable (cf. Theorem \ref{simpleNormal}).  We also determine the cohomology  and hence the depth of the conormal sheaf of $X$ (cf. Proposition \ref{depthconormal} and Corollaries \ref{vanis2} and \ref{vanis3}).
 Finally, in section 5, we face Problem  \ref{pblm} (4) and we prove that  the normal sheaf    ${\mathcal N}_X$ of a standard determinantal scheme $X\subset \PP^n$ of codimension $c$ associated to a $t\times (t+c-1)$ matrix with linear entries is  always $\mu$-semistable (cf. Theorem \ref{semistablenormal}) and even more it is $\mu$-stable when $c=2$ and $n\ge 4$ (cf. Theorem \ref{stablenormal}).

 \vskip 2mm \noindent {\bf Remark.} This version on the arXiv makes a
 correction to Proposition 3.9 of previous versions on the arXiv, as well as
 to the published version in Crelle's journal, see Remark~\ref{correctionC14}
 for details.

 \vskip 2mm \noindent {\bf Acknowledgement.} The second author would like to thank A. Conca, L. Costa and J. Pons-Llopis for useful discussions on the subject. She also thanks Oslo and Akershus University College for  its hospitality during her visit to Oslo in October 2011. The first author thanks the University of Barcelona for  its hospitality during his visit to Barcelona in June 2012.

 We also thank the referees for their comments and for pointing out that our first formulation of Lemma 3.1 was inaccurate.

\vskip 2mm
\noindent \underline{Notation.} Throughout this paper $K$ will be an algebraically closed field of characteristic zero, $R=K[x_0, x_1, \cdots ,x_n]$, $\mathfrak{m}=(x_0, \ldots,x_n)$ and $\PP^n=\Proj(R)$. Given a closed subscheme $X\subset \PP^n$, we denote by  ${\mathcal I}_X$ its ideal sheaf,  ${\mathcal
N}_X$ its normal sheaf and $ I(X)=\Hl^0_{*}(\PP^n, {\mathcal I}_X)$
its saturated homogeneous ideal unless $X=\emptyset $, in which case we let $I(X)=\mathfrak{m}$. If $X$ is equidimensional and
Cohen-Macaulay of codimension $c$, we set $\omega_X ={\mathcal
E}xt^c_{{\mathcal O}_{\PP^n}} ({\mathcal O}_X,{\mathcal
O}_{\PP^n})(-n-1)$ to be its canonical sheaf. Given a  coherent sheaf $\shE$ on $X$ we denote the twisted sheaf $\shE\otimes\odi{X}(l)$ by $\shE(l)$. As usual, $\Hl^i(X,\shE)$ stands for the cohomology groups, $\h^i(X,\shE)$ for their dimension and $\Hl^i_*(X,\shE)=\oplus _{l \in \ZZ}\Hl^i(X,\shE(l))$.

For any graded quotient $A$ of $R$ of codimension $c$,
we let $I_A=\ker(R\twoheadrightarrow A)$ and we let $N_A=\Hom_R(I_A,A)$ be
the normal module. If $A$ is Cohen-Macaulay of codimension $c$, we
let  $K_A=\Ext^c_R (A,R)(-n-1)$ be its canonical module. When we
write $X=\Proj(A)$, we let $A=R_X:=R/I(X)$ and $K_X=K_A$. If $M$ is a
finitely generated graded $A$-module, let $\depth_{J}{M}$ denote
the length of a maximal $M$-sequence in a homogeneous ideal $J$
and let $\depth {M} = \depth_{\mathfrak m}{ M}$.


\section{Preliminaries}
For convenience of the reader we include in this section the
background and basic results on (linear) standard determinantal varieties as well as on arithmetically Cohen-Macaulay sheaves and Ulrich sheaves needed in the sequel.

\subsection{Determinantal varieties}

Let us start recalling the definition of
arithmetically Cohen-Macaulay subschemes of $\PP^n$.

\begin{defn}\rm  A
subscheme  $X\subset \PP^n$  is said to be arithmetically
Cohen-Macaulay (briefly, ACM)  if
 its homogeneous
coordinate ring $R/I(X)$ is a Cohen-Macaulay ring, i.e. $\depth
R/I(X)=\dim R/I(X)$.
\end{defn}

Thanks to the graded version of the
Auslander-Buchsbaum formula (for any finitely generated
$R$-module $M$): $$\pd M=n+1-\depth M,$$ we deduce that a
subscheme $X\subset \PP^n$ is ACM if and only if
$ \pd R/I(X)=\codim X.$
Hence, if $X\subset \PP^n$ is a codimension $c$ ACM subscheme, a
graded minimal free $R$-resolution of $I(X)$ is of the form:
$$ 0\rightarrow F_c   \stackrel{\varphi _c}{\longrightarrow}  F_{c-1}\stackrel{\varphi _{c-1}}{\longrightarrow}  \cdots
\stackrel{\varphi _2}{\longrightarrow}  F_1 \stackrel{\varphi _1}{\longrightarrow}  I(X)\rightarrow 0$$ where
$F_{i}=\oplus _{j\in \ZZ}R(-j)^{\beta_{i,j}}$, $i=1, \cdots ,c $
(in this setting, minimal means that $\im (\varphi _{i})\subset
\mathfrak{m}F_{i-1}$).

We now collect the results on
standard
determinantal schemes and the associated complexes needed  in the sequel and we refer to \cite{b-v} and \cite{eise} for more details.

\begin{defn}
\rm If  $\cA$ is a homogeneous matrix, we denote by  $I(\cA)$ the
ideal of $R$ generated by the  maximal minors of $A$ and by $I_j(\cA)$
the ideal generated by the $j \times j$ minors of $\cA$. A codimension $c$
subscheme
$X\subset
\PP^n$ is called a
\emph{standard determinantal} scheme if $I(X)=I(\cA)$ for some
$t\times (t+c-1)$ homogeneous matrix $\cA$. In addition, we will say that  $X$  is  a \emph{linear standard  determinantal scheme} if all entries of $\cA$ are linear forms.
 \end{defn}

 As examples of standard determinantal schemes we have any ACM variety $X\subset \PP^n$ of codimension 2 and as examples of linear determinantal schemes we have, for instance, rational normal curves, Segre varieties and rational normal scrolls.
 Now we are going to describe some complexes associated to a
codimension $c$ standard determinantal scheme $X$. To this end, we
denote by
$\varphi
:F\longrightarrow G$ the morphism of free graded $R$-modules of
rank $t+c-1$ and $t$, defined by the homogeneous matrix $\cA$ of $X$.
We denote by ${\cC}_i(\varphi )$ the (generalized) Koszul complex:
$${\cC}_i(\varphi): \  0 \rightarrow \wedge^{i}F \otimes S_{0}(G)\rightarrow
\wedge^{i-1} F
\otimes S
_{1}( G)\rightarrow \ldots \rightarrow \wedge ^{0} F \otimes S_i (G) \rightarrow 0$$
where $\wedge ^{i}F$ stands for the $i$-th exterior power of $F$ and $S_i (G)$
denotes the $i$-th symmetric power of $G$.
 Let ${\cC}_i(\varphi)^*$ be the $R$-dual of
${\cC}_i(\varphi)$ . The dual map $\varphi^*$ induces graded
morphisms $$\mu_i:\wedge ^{t+i}F\otimes
\wedge^tG^*\rightarrow \wedge^{i}F.$$
They can be used to splice the complexes ${\cC}_{c-i-1}(\varphi)^*\otimes \wedge^{t+c-1}F\otimes \wedge^tG^*$ and
 ${\cC}_i(\varphi)$ to a complex ${\cD}_i(\varphi):$
 \begin{equation}\label{splice}
0 \rightarrow \wedge^{t+c-1}F \otimes S_{c-i-1}(G)^*\otimes
\wedge^tG^*\rightarrow \wedge^{t+c-2} F \otimes S _{c-i-2}(G)^*\otimes
\wedge ^tG^*\rightarrow \ldots \rightarrow \end{equation}
$$\wedge^{t+i}F \otimes
S_{0}(G)^*\otimes \wedge^tG^*\stackrel{\epsilon _{i}}{\longrightarrow} \wedge^{i} F \otimes S _{0}(G)
\rightarrow \wedge ^{i-1} F \otimes S_1(G)\rightarrow \ldots \rightarrow \wedge^0F\otimes
S_i(G)\rightarrow 0 .$$
The complex  ${\cD}_0(\varphi)$ is called the Eagon-Northcott
complex and the complex  ${\cD}_1(\varphi)$ is called the
Buchsbaum-Rim complex. Let us rename the complex  ${\cC}_c(\varphi)$ as  ${\cD}_c(\varphi)$. Denote by $I_m(\varphi)$ the
ideal generated by the $m\times m$ minors of the matrix $\cA$
representing $\varphi$. Then, we have the following well known result:

\begin{proposition}\label{resol} Let $X \subset \PP^n$ be a standard determinantal subscheme
of codimension $c$ associated to a graded minimal (i.e.
$\im(\varphi)\subset \mathfrak{m}G$) morphism $\varphi: F\rightarrow G$ of free
$R$-modules of rank $t+c-1$ and $t$, respectively. Set $M= Coker(
\varphi )$. Then, it holds:

(i) ${\cD}_i(\varphi)$ is acyclic for $-1\le i \le c$.

(ii) ${\cD}_0(\varphi)$ is a minimal free graded $R$-resolution of
$R/I(X)$  and  ${\cD}_i(\varphi)$ is a minimal free graded
$R$-resolution of length $c$ of $S_i(M)$, $1\le i \le c$ .

 (ii) Up to twist, $K_X \cong
S_{c-1}(M)$. So, up to twist,  $\mathcal{D}_{c-1}(\varphi)$ is a minimal
free graded $R$-module resolution of $K_X$.

(iv) ${\cD}_i(\varphi)$ is a minimal free graded $R$-resolution of
$S_i(M)$ for $c+1\le i$ whenever $depth_{I_m(\varphi)}R\ge
t+c-m$ for every $m$ such that $t\ge m \ge max(1,t+c-i)$.
\end{proposition}

\begin{proof}
See, for instance \cite{b-v}; Theorem 2.20  and \cite{eise}; Theorem A2.10 and Corollaries A2.12 and
A2.13.
\end{proof}

It immediately follows from Proposition 2.3 (ii) that standard determinantal schemes are ACM.
Let us also recall
 the following useful comparison of cohomology
groups. If $Z\subset X$ is a closed subset such that $U=X\setminus Z$ is a local complete intersection,  $L$ and $N$ are finitely generated $R/I(X)$-modules, $\widetilde{N}$ is locally free on
$U$ and
 $\depth_{I(Z)}L\ge r+1$, then the natural map
\begin{equation} \label{NM}
\Ext^{i}_{R/I(X)}(N,L)\longrightarrow
\Hl_{*}^{i}(U,{\mathcal H}om_{\odi{X}}(\widetilde{N},\widetilde{L}))
\end{equation}
is an isomorphism, (resp. an injection) for $i<r$ (resp. $i=r$)
cf. \cite{SGA2}, expos\'{e} VI. Recall that we interpret $I(Z)$ as $\mathfrak{m}$ if $Z=\emptyset $.

 We end this subsection describing the Picard group of a smooth standard determinantal scheme $X$.
 Assume that $X\subset \PP^n$ is given by the
maximal minors of a $t \times (t+c-1)$  homogeneous matrix $\cA$
representing a homomorphism $\varphi$ of free graded R-modules $$
\varphi: F= \bigoplus_{i=1}^{t+c-1} R(-a_i) \longrightarrow G=
\bigoplus_{j=1}^{t} R(-b_j) .  $$ Without loss of generality, we may
assume $a_1 \leq a_2 \leq \ldots \leq a_{t+c-1}$ and $b_1 \leq b_2
\leq \ldots \leq b_t$. Denote by $H$ the general hyperplane section of $X$ and by $Z\subset X$ the codimension 1 subscheme of $X$ defined by the maximal minors of the $(t-1)\times (t+c-1)$ matrix obtained deleting the last row of $\cA$. The following theorem computes the Picard
group of $X$. Indeed, we have:

\begin{theorem}\label{picard}  Let $X\subset \PP^n$ be a smooth standard
determinantal scheme of codimension $c\ge 2$. Set $\ell :=\sum _{j=1}^{t+c-1}a_j-\sum _{i=1}^tb_i$. Assume $t>1$. If  $n-c> 2$ and  $a_1-b_{t}>0$; or $n-c=2$,  $a_1-b_{t}>0$ and $\ell \ge n+1$, then
 $Pic(X) \cong \mathbb{Z}^2\cong \langle H,Z \rangle$.
\end{theorem}

\begin{proof}
See \cite{Ein}; Corollary 2.4 for smooth standard determinantal varieties
$X\subset \PP^n$ of dimension $d\ge 3$ and \cite{FF};  Proposition 5.2 for the case $d=2$ (see also \cite{Lo}; Theorem II.4.2,
for smooth surfaces $X\subset \PP^4$).
\end{proof}

\subsection{ACM and Ulrich sheaves}

 We set up here some preliminary notions mainly concerning  the definitions and basic results on  ACM and Ulrich sheaves  needed later.

\begin{defn}\label{ACM} \rm
Let $X\subset \PP^n$ be a projective scheme and let $\shE$ be a coherent sheaf on $X$. $\shE$  is said to be \emph{Arithmetically Cohen Macaulay} (shortly, ACM) if it is locally Cohen-Macaulay (i.e., $\depth \shE_x=\di \odi{X,x}$ for every point $x\in X$) and has no intermediate cohomology, i.e. $$
\Hl^i(X,\shE (t))=0 \quad\quad \text{    for all $t$ and $i=1, \ldots , \di X-1.$}
$$
\end{defn}
Notice that when $X$ is a non-singular variety, any coherent ACM sheaf on $X$ is locally free. A seminal result  due to
Horrocks (cf. \cite{Hor})  asserts that, up to twist, there is only one indecomposable ACM
bundle on $\PP^n$: $\odi{\PP^n}$.
Ever since this result was stated, the study of the category of indecomposable arithmetically Cohen-Macaulay bundles  on a given projective scheme $X$ has raised a lot of interest since it is a natural way to understand the complexity of the underlying variety $X$ (for more information the reader can see \cite{CH}, \cite{CMP}, \cite{MR}, \cite{MP2012} and \cite{MP}). One of the goals of this paper is to study whether the normal sheaf
${\mathcal N}_X$ and the exterior power $\wedge ^q{\mathcal N}_X$ of the normal sheaf
(or suitable twists ${\mathcal N}_X\otimes \cL $ and  $\wedge ^q{\mathcal N}_X\otimes
\cL'$
by  invertible sheaves $\shL$ and $\shL'$)
of a standard determinantal scheme $X\subset \PP^n$ are ACM. ACM sheaves are closely
related to their algebraic counterpart, the maximal Cohen-Macaulay modules.

\begin{defn} \rm
A graded $R_X$-module $E$ is a \emph{maximal Cohen-Macaulay} module (MCM for short) if $\depth E=\di E=\di R_X$.
\end{defn}

In fact, we have:

\begin{proposition}\label{bijection}
Let $X\subseteq\PP^n$ be an ACM scheme. There exists a bijection between ACM sheaves $\shE$ on $X$ and MCM $R_X$-modules $E$ given by
the functors $E\rightarrow \widetilde{E}$ and $\shE\rightarrow \Hl^0_*(X,\shE)$.
\end{proposition}

\begin{proof}
See \cite{CH04}; Proposition 2.1.
\end{proof}

\begin{defn}\rm
Given a closed subscheme $X\subset \PP^n$, a coherent sheaf $\shE$ on  $X$
  is said to be \emph{initialized} if
$$
\Hl^0(X,\shE(-1))=0 \ \ \text{ but } \  \Hl^0(X,\shE)\neq 0.
$$
Notice that when $\shE$ is a locally Cohen-Macaulay sheaf, there always exists an integer $k$ such that $\shE_{init}:=\shE(k)$ is
initialized.
\end{defn}

Let us now introduce the notion of  Ulrich sheaf.
\begin{defn}
Given a projective scheme $X\subset \PP^n$ and a coherent sheaf $\shE$ on $X$, we say that $\shE$ is an \emph{Ulrich sheaf} if  $\shE$ is an ACM sheaf and $\h^0(\shE_{init})=\deg(X)\rk(\shE)$.
\end{defn}

We have the following result that justifies this definition:
\begin{theorem}
Let $X\subseteq\PP^n$ be an integral subscheme and let $\shE$  be an ACM sheaf on $X$. Then the minimal number of generators $m(\shE)$ of the $R_X$-module $\Hl^0_*(\shE)$ is bounded by
$$
m(\shE)\leq \deg (X)\rk(\shE).
$$
\end{theorem}

Therefore, since it is obvious that for an initialized sheaf $\shE$, $\h^0(\shE)\leq m(\shE)$, the minimal number of generators of Ulrich sheaves is
as large as possible. Modules attaining this upper bound were studied by Ulrich in \cite{Ulr}. A detailed account is provided in \cite{ESW}. In particular we have:

\begin{theorem}\label{equivconditionsulrich}
Let $X\subseteq\PP^N$ be an $n$-dimensional ACM variety and $\shE$  be an initialized ACM coherent sheaf on $X$. The following conditions are equivalent:
\begin{enumerate}
\item[(i)] $\shE$ is Ulrich.
\item[(ii)] $\shE$ admits a linear $\odi{\PP^N}$-resolution of the form:
$$
0 \arr \odi{\PP^N}(-N+n)^{a_{N-n}}\arr \dots \arr \odi{\PP^N}(-1)^{a_1}\arr\odi{\PP^N}^{a_0}\arr\shE\arr 0.
$$

\item[(ii)] $\Hl^i(\shE(-i))=0$ for $i>0$ and $\Hl^i(\shE(-i-1))=0$ for $i<n$.
\item[(iv)] For some (resp. all) finite linear projections $\pi:X\arr\PP^n$, the sheaf $\pi_*\shE$ is the trivial sheaf $\odi{\PP^n}^t$ for some $t$.
\end{enumerate}
In particular, initialized Ulrich sheaves are $0$-regular and therefore they are globally generated.
\end{theorem}
\begin{proof} See  \cite{ESW} Proposition 2.1.
\end{proof}

In the next section, we will prove that under some  weak restrictions the normal sheaf ${\mathcal N}_X$ of a linear standard determinantal scheme $X$ twisted by a suitable invertible sheaf   $\cL$ on
$X\setminus Sing(X)$ is Ulrich (cf. Theorem \ref{normalresolution}); and as an immediate consequence we will deduce the $\mu$-semistability of the normal sheaf of a linear standard determinantal scheme (cf. Theorem \ref{semistablenormal}).


\section{The Cohen-Macaulayness of the normal sheaf of a determinantal variety}

The main goal of this  section is  to answer Problem \ref{pblm}(1). More precisely, we prove that, under some mild assumptions, given a standard determinantal scheme $X\subset \PP^n$ of codimension $c$ there always exists an invertible sheaf $\shL$ on $X\setminus Sing(X)$ such
${\mathcal N}_X\otimes \shL$ is ACM (cf. Theorem \ref{normalresolution}) and as a by-product we obtain an  $\odi{\PP^n}$-resolution of ${\mathcal
N}_X\otimes \shL$  which turns out to be pure and linear in the case of linear standard determinantal schemes (cf. Theorem \ref{normalresolution} and Corollary \ref{res}).
 At the end of this section, we conjecture an analogous
result for the exterior power of the normal sheaf (cf. Conjecture \ref{conjetura2}).

\vskip 2mm In this section $X\subset \PP^n$ will be a standard determinantal
scheme of codimension $c$, $\cA $ the $t \times (t+c-1)$ homogeneous matrix
associated to $X$, $I=I_t(\cA)$, $A=R/I$, $N_A:=\Hom_A(I/I^2,A)$, ${\mathcal
  N}_X:=\widetilde{N_A}$, $$\varphi:F:=\bigoplus _{j=1}^{t+c-1}
R(-a_j)\longrightarrow G:=\bigoplus _{i=1}^tR(-b_{i})$$ the morphism of free
$R$-modules associated to $\cA$ and $M:=\coker(\varphi )$. We will assume
$t>1$, since the case $t=1$ corresponds to a codimension $c$ complete
intersection $X\subset \PP^n$ and its normal sheaf is well understood
(${\mathcal N}_X\cong \oplus _{i=1}^c \odi{X}(d_i)$). If $t\ge 2$,
$\widetilde{M}$ is a locally free $\odi{X}$-module of rank 1 over
$T:=X\setminus V(J)$ where $J:=I_{t-1}(\cA)$ and $T\hookrightarrow \PP^n$ is a
local complete intersection. Recall also that if $a_j > b_i$ for any $i,j$,
then $V(J)=Sing(X)$, $\codim _X(Sing(X))=c+2$ or $Sing(X)=\emptyset $ and $\codim _{\PP^n}X=c$ for a
general choice of $\varphi \in \Hom (F,G)$.

Recall that the normal sheaf ${\mathcal N}_X$ of a standard determinantal
scheme is ACM if and only if $\Hl ^{i}_*({\mathcal N}_X)=0$ for $0<i<\dim X$.
In particular, the normal sheaf ${\mathcal N}_X$ of a standard determinantal
curve $X\subset \PP^n$ ($c=n-1$) is always ACM as well as the normal sheaf
${\mathcal N}_X$ of a standard determinantal scheme $X\subset \PP^n$ of
codimension $1\le c\le 2$. Indeed, if $c=1$, then ${\mathcal N}_X\cong
\odi{X}(\delta )$ where $\delta:=\deg(X)$ and if $c=2$, then there is an exact
sequence (cf. \cite{KP}, (26))
\begin{equation} \label{Normalcodim2}
 0 \rightarrow \widetilde{F} \otimes_R \widetilde{G}^* \rightarrow  (( \widetilde{F}^* \otimes_R \widetilde{F})
 \oplus (\widetilde{G}^* \otimes_R \widetilde{G}))/R \rightarrow
 \widetilde{G} \otimes_R \widetilde{F}^* \rightarrow {\mathcal N}_X \rightarrow 0.
\end{equation}
Unfortunately, it is no longer true for higher codimension and we only have
that under some mild conditions $\Hl^{i} _*({\mathcal N}_X)=0$ for $1\le i \le
n-c-2$ (see \cite{KP2}; Lemma 35 for $c=3$, \cite{KM05}; Corollary 5.5 for
$3\le c\le 4$ and \cite{K2011}; Theorem 5.11 for the general case). In this
section we are going to prove that under some weak restrictions ${\mathcal
  N}_X\otimes \widetilde{S_{c-2}M}$ is ACM, i.e.
$$\Hl^{i} _*(X,{\mathcal N}_X\otimes \widetilde{S_{c-2}M})=0 \text{ for } 1\le i \le n-c-1.$$

To prove it, we will use the following technical lemma which can be seen as a
generalization of the mapping cone process and we include a proof for the sake
of completeness. In this lemma the differentials of a complex, say
$Q_{\bullet}$, are denoted by $d_Q^{i}:Q_i \to Q_{i-1}$.

 \begin{lemma}\label{technicallemma}
   Let $Q_{\bullet} \stackrel{\sigma _{\bullet}}{\longrightarrow} P _{\bullet}
   \stackrel{\tau _{\bullet}}{ \longrightarrow} F _{\bullet} $ be morphisms of
   complexes satisfying $Q_j=P_j=F_j=0$ for $j < 0$ and assume that all three
   complexes are acyclic (exact for $j \ne 0$) and that the sequence
\begin{equation*}
  0 \longrightarrow \coker d_Q^{1}  \longrightarrow \coker d_P^{1} \stackrel {\alpha}
  { \longrightarrow}  \coker d_F^{1}
\end{equation*}
is exact. Moreover assume that there exists a morphism $\ell _{\bullet}: Q_{\bullet}\longrightarrow  F_{\bullet}[1]$ such that
for any integer $i$:
\begin{equation} \label{ellcomp} d_F^{i+1}\ell _i+\ell _{i-1}d_Q^{i}=\tau
  _i\sigma_i.
\end{equation}
Then, the complex  $Q_{\bullet}\oplus P_{\bullet}[1]\oplus F_{\bullet}[2]$ given by
$$Q_{i}\oplus P_{i+1}\oplus F_{i+2} \stackrel {d^{i}_{Q,P,F}}
{ \longrightarrow} Q_{i-1}\oplus P_{i}\oplus F_{i+1}$$ where
 $$d^{i}_{Q,P,F}:=\begin{pmatrix}d^{i}_Q & 0 & 0 \\ \sigma _{i} & -d_P^{i+1} & 0 \\
 \ell _{i} & -\tau_{i+1} & d^{i+2}_F \end{pmatrix} $$ is acyclic (exact for $i \ne -2$) and $ \coker  {d^{-1}_{Q,P,F}} = \coker \alpha$.
 \end{lemma}
 \begin{proof} It is straightforward to show that $ {d^{i-1}_{Q,P,F}} \circ
   {d^{i}_{Q,P,F}}=0$ by using $\eqref{ellcomp}$ and that the differentials of
   the complexes commute with $\sigma _{\bullet}$ and $\tau_{\bullet}$.
   To see the acyclicity of  $Q_{\bullet}\oplus P_{\bullet}[1]\oplus F_{\bullet}[2]$, let  $$\begin{pmatrix}d^{i-1}_Q & 0 & 0 \\ \sigma _{i-1} & -d_P^{i} & 0 \\
     \ell _{i-1} & -\tau_{i} & d^{i+1}_F \end{pmatrix}\begin{pmatrix} q \\ p \\
     f \end{pmatrix} =\begin{pmatrix}d^{i-1}_Q(q) \\ \sigma _{i-1}(q) -d_P^{i}(p) \\
     \ell _{i-1}(q) -\tau_{i}(p) + d^{i+1}_F(f) \end{pmatrix}=\begin{pmatrix} 0 \\ 0 \\
     0 \end{pmatrix} \ .$$ We must show the existence of $(q',p',f') \in
   Q_{i}\oplus P_{i+1}\oplus F_{i+2}$ whose transpose maps to $(q,p,f)^{\rm
     tr}$ via $d^{i}_{Q,P,F}$. Since $Q_{\bullet}$ is exact for $i \ne 0$
   there exists a $q' \in Q_{i}$ such that $d^{i}_{Q}(q')=q$ provided $i > 1$.
   For $i =1$ we use $ \sigma _{0}(q) = d_P^{1}(p) $ and the injectivity of $
   \coker d_Q^{1} \rightarrow \coker d_P^{1}$ to see that $q $ maps to zero in
   $ \coker d_Q^{1}$. Hence there exists $q' \in Q_{1}$ such that
   $d^{1}_{Q}(q')=q$. We will treat the case $i=0$ shortly.

   Suppose $i > 0$. Since we have $ \sigma _{i-1}(q) -d_P^{i}(p)=0$,
   $q=d^{i}_{Q}(q')$ and $\sigma _ {i-1}d^{i}_{Q}=d^{i}_{P}\sigma _{i}$ we get
   $ d^{i}_{P}(\sigma_{i }(q') - p)=0$. By the exactness of $P_{\bullet}$,
   there exists an element $p' \in P_{i+1}$ such that
   $d^{i+1}_{P}(p')=\sigma_{i }(q') - p$. Then we insert $p= \sigma_{i }(q')
   -d^{i+1}_{P}(p')$ and $q=d^{i}_{Q}(q')$ into $ \ell _{i-1}(q) -\tau_{i}(p)
   + d^{i+1}_F(f)=0$ and we use \eqref{ellcomp} and that $d^{i+1}_{P}$ and
   $d^{i+1}_{F}$ commute with $\tau_{\bullet}$ to conclude that $d^{i+1}_{F}(-
   \ell _{i}(q') +\tau_{i+1}(p') + f)=0$. Hence there exists an element $f'
   \in F_{i+2}$ such that $ f=\ell _{i}(q') -\tau_{i+1}(p') +
   d^{i+2}_{F}(f')$. The expressions for $q, p$ and $f$ above prove precisely
   what we needed to show if $i > 0$.

   Now suppose $i = 0$. Since $d^{-1}_{Q}=d^{0}_{Q}=d^{0}_{P}=
   \sigma_{-1}=\ell_{-1}=0$ and $q \in Q_{-1} = 0$ we have an element $(p,f)
   \in P_{0}\oplus F_{1}$ that maps to $-\tau_0(p)+d^1_F(f) = 0$ in $F_0$.
   Since $\tau_0(p)$ maps to zero in $\coker d_F^{1}$, it follows that $p$ is
   sent to an element in $ \coker d_P^{1}$ that is contained in $\coker
   d_Q^{1}$ by the exactness assumption on the sequence of cokernels of
   Lemma~\ref{technicallemma}. Hence there exists an element $q' \in Q_0$ such
   that $\sigma_0(q')-p$ maps to zero in $ \coker d_P^{1}$, whence there is a
   $p' \in P_1$ such that $d^{1}_{P}(p')=\sigma_0(q')-p$. Since we have
   $d^1_F(-l_0(q')+\tau_1(p')+f)=0$ there exists $f' \in F_2$ such that
   $d^2_F(f')=-l_0(q')+\tau_1(p')+f$, and we get the expressions for $p$ and $f$
   which we wanted to show.

   Finally using that $(p,f) \in P_{0}\oplus F_{1}$ maps to
   $-\tau_0(p)+d^1_F(f) $ in $F_0$ we easily get $ \coker {d^{-1}_{Q,P,F}} =
   \coker \alpha$ and we are done.
\end{proof}

\begin{theorem} \label{normalACM} We keep the notation introduced above and, in addition,  we assume $\depth _JA\ge 2$. Then, we have  $$\Ext^1_R(M,S_{c-1}M)$$ is a maximal Cohen-Macaulay
$A$-module of rank $c$. Moreover,  if $a_j=1$ for all $j$ and $b_i=0$ for all $i$, then $\widetilde{\Ext^1_R(M,S_{c-1}M)}$ is an Ulrich sheaf of rank $c$.
\end{theorem}

\begin{proof}
Our  primary goal is to show that $\Ext^1_R(M,S_{c-1}M)$ is a MCM
$A$-module by using the exact sequence
\begin{equation}\label{mainseq}
 0 \longrightarrow S_{c-2}M   \longrightarrow G^*\otimes S_{c-1}M  \longrightarrow F^*\otimes S_{c-1}M \longrightarrow \Ext^1
_R(M, S_{c-1}M)  \longrightarrow 0.
\end{equation}
 and Lemma \ref{technicallemma} to exhibit a minimal free resolution of $\Ext^1_R(M,S_{c-1}M)$ having length $c$.
So let us start  proving the existence of the exact sequence (\ref{mainseq}).
Indeed, we look at the Buchsbaum-Rim complex (see Proposition \ref{resol}(ii))
 \begin{equation}\label{BR} \cdots  \longrightarrow \wedge ^{t+1}F\otimes \wedge ^tG^*\otimes S_0G^*  \stackrel{\epsilon _1^*}{ \longrightarrow} F \stackrel{\varphi }{ \longrightarrow}G  \longrightarrow M \longrightarrow 0.\end{equation}
Since $I_t(\varphi)\cdot M=0$ and $\im(\epsilon _1^*) \subseteq I_t(\varphi)\cdot F$, we get
that the induced map $\Hom _R(\epsilon _1^*, S_{c-1}M) = 0$.
Therefore, the exact sequence (\ref{mainseq}) comes from applying the functor  $\Hom_ R(-,S_{c-1}M)$ to the exact sequence  (\ref{BR}),
 because under the assumption $\depth _J A\ge 2$, we have  $\depth _J  S_{i}M \ge 2$ for any $1 \le i \le c$ and hence
\begin{equation}\label{Hom(M,SiM)}
\begin{array}{cll}
\Hom_R(M,S_{i}M) & \cong & \Hl ^0_*(X\setminus V(J),{\mathcal H}om_{\odi{X}}(\widetilde{M},\widetilde{M}^{\otimes (i)}))\\ & \cong &
 \Hl ^0_*(X\setminus V(J),\widetilde{M}^{\otimes (i-1)}) \\ & \cong & S_{i-1}M.\end{array}
\end{equation}

From the exact sequence (\ref{mainseq}) we deduce that $\Ext^1
_R(M, S_{c-1}M)$ has rank $c$; let us prove that it is a maximal Cohen-Macaulay $A$-module. The idea will be
 be to apply Lemma \ref{technicallemma} to the following diagram  which we will define as an expansion of (\ref{mainseq}) (we set $\wedge ^{i}:=\wedge^{i}F$):

 $$
 {\scriptsize
 \begin{array}{cccccccccccccc}& & 0 &  & 0 & & 0 & &&&&& \\& &\downarrow &  & \downarrow & &  \downarrow & &&&&& \\
  &  &  \wedge^{t+c-1}\otimes S_{1}G^*\otimes \wedge^t G^*  & \stackrel{\sigma _c }{ \dashrightarrow } &  G^*\otimes \wedge^{t+c-1}\otimes S_{0}G^*\otimes \wedge ^{t}G^* & \stackrel{\varphi^*\otimes 1 }{ \longrightarrow } & F^*\otimes \wedge^{t+c-1}\otimes S_{0}G^*\otimes \wedge ^{t} G^* &  &  &  & & &
   \\
 & &  \downarrow &  & \downarrow & &  \downarrow & &&&&& \\
  &   &  \wedge^{t+c-2}\otimes S_{0}G^*\otimes \wedge^t G^*  & \stackrel{\sigma _{c-1} }{ \dashrightarrow } &  G^*\otimes S_{0}G\otimes \wedge ^{c-1} & \stackrel{\varphi^*\otimes 1}{ \longrightarrow } & F^*\otimes S_{0}G\otimes \wedge ^{c-1}  &  &  &  & & &\\
  & & \downarrow &  & \downarrow & &  \downarrow & &&&&&& \\
   & &  S_{0}G\otimes \wedge ^{c-2}    & \stackrel{\sigma _{c-2} }{ \dashrightarrow } &  G^*\otimes S_{1}G\otimes \wedge ^{c-2}  & \stackrel{\varphi^*\otimes 1 }{ \longrightarrow } & F^*\otimes S_{1}G \otimes \wedge^{c-2} &  &  &  & & &\\
 & &\downarrow &  & \downarrow & &  \downarrow & &&&&&& \\
&  &\vdots &  & \vdots & &  \vdots & &&&&&& \\
 &  &\downarrow &  & \downarrow & &  \downarrow & &&&&&& \\
 &   &    S_{c-4}G\otimes \wedge^2   & \stackrel{\sigma _2 }{ \dashrightarrow } &  G^*\otimes S_{c-3}G\otimes \wedge^2  & \stackrel{\varphi^*\otimes 1 }{ \longrightarrow } & F^*\otimes S_{c-3}G\otimes \wedge^2 &  &  &  & & &\\
 &  &\downarrow &  & \downarrow & &  \downarrow & &&&&&& \\
  &  &   S_{c-3}G\otimes F  & \stackrel{\sigma _1 }{ \dashrightarrow } &  G^*\otimes S_{c-2}G\otimes F  & \stackrel{\varphi^*\otimes 1 }{ \longrightarrow } & F^*\otimes S_{c-2}G\otimes F &  &  &  & & &\\
   & &\downarrow &  & \downarrow & &  \downarrow & &&&&&& \\
    & &  S_{c-2}G  & \stackrel{\sigma _0 }{ \dashrightarrow } &  G^*\otimes S_{c-1} G & \stackrel{\varphi^*\otimes 1 }{ \longrightarrow }& F^*\otimes S_{c-1}G &  &  &  & & &\\
   & &\downarrow &  & \downarrow & &  \downarrow & &&&&&& \\
  0   & \rightarrow  & S_{c-2}M  & \rightarrow &  G^*\otimes S_{c-1}M  & \stackrel{\varphi^*\otimes 1 }{ \longrightarrow } & F^*\otimes S_{c-1}M & \, \hskip -1cm \rightarrow   \Ext^1
_R(M, S_{c-1}M)  \rightarrow  0\\
& & \downarrow &  & \downarrow & &  \downarrow & &&&&&& \\
& & 0 &  & 0 & & 0 & &&&&&
 \end{array} }
$$
$$ {\scriptsize \text{Diagram A}}
$$

Let us call $Q_{\bullet}$, $P_{\bullet}$ and $F_{\bullet}$ the resolutions of
$S_{c-2}M$, $G^*\otimes S_{c-1}M$ and $F^*\otimes S_{c-1}M$, respectively.
We need  to define morphisms of complexes:
$\sigma _{\bullet}: Q_{\bullet}\rightarrow P_{\bullet}$ and
$\ell _{\bullet}: Q_{\bullet}\rightarrow F_{\bullet}[1]$ satisfying all the hypothesis of Lemma \ref{technicallemma}. Let us first recall the definition of
$$\partial _q^p:S_pG\otimes \wedge ^qF \longrightarrow
S_{p+1}G\otimes \wedge ^{q-1}F.$$
To this end, we take $\{x_{i} \}_{i=1}^t$ an $R$-free basis of $G$ and let $\{x_{i}^* \}_{i=1}^t$ be the dual basis of $G^*$. According to \cite{eise}; pg. 592, $\partial _q^p$ takes an element $m\otimes f \in S_pG\otimes \wedge ^qF$ to the element $\sum _{i=1}^tx_{i}m\otimes \varphi^*(x_{i}^*)(f)\in S_{p+1}G\otimes \wedge ^{q-1}F$. For any integer $i$, $0\le i \le c-2$, we  define:  $$\sigma _{i}: S_{c-2-i}G\otimes \wedge ^{i}F      \longrightarrow    G^*\otimes S_{c-1-i}G\otimes \wedge ^{i}F $$ sending an element $m\otimes f \in S_{c-2-i}G\otimes \wedge ^{i}F$ to $\sigma _{i}(m\otimes f)=\sum _{j=1}^tx_{j}^*\otimes x_{j}m\otimes f\in G^*\otimes S_{c-1-i}G\otimes \wedge ^{i}F$.
It is  easy  to check that the following diagram commutes for any integer $i$, $0\le i \le c-3$:
 \[
 \begin{array}{ccccc}
    S_{c-3-i}G\otimes \wedge ^{i+1}F &  \stackrel{\sigma _{i+1} } {\longrightarrow } &  G^*\otimes S_{c-2-i}G\otimes \wedge ^{i+1}F \\
\downarrow  \scriptstyle{\partial ^{c-3-i}_{i+1}}& &  \downarrow \scriptstyle{
 1_{G^*}\otimes \partial ^{c-2-i}_{i+1}}\\ S_{c-2-i}G\otimes \wedge ^{i}F
  &  \stackrel{\sigma _{i} }  {\longrightarrow } &  G^*\otimes S_{c-1-i}G\otimes \wedge ^{i}F
 \end{array}.
\]
In this setting, we point out that  the definition of $\sigma _{c-2}$ implies the commutativity of the diagram
 \[
 \begin{array}{ccccc}
    S_{0}G\otimes \wedge ^{c-2}F &  \stackrel{\sigma _{c-2} } {\longrightarrow } &  G^*\otimes S_{1}G\otimes \wedge ^{c-2}F \\
\downarrow \simeq & &  \downarrow \simeq \\ R\otimes \wedge ^{c-2}F
  &  \stackrel{\text{ tr}\otimes 1
   }  {\longrightarrow } &  G^*\otimes G\otimes \wedge ^{c-2}F
 \end{array}
\]
i.e., $\sigma _{c-2}$ is induced by the trace map $R\longrightarrow G^*\otimes
G$ that is dual to the evaluation map $G^*\otimes G\longrightarrow R$ (i.e.
$\sigma _{c-2}={\rm tr}\otimes 1$). We will now define $\sigma _c$ and $\sigma
_{c-1}$ in such a way that the two left upper squares of the Diagram A commute
or anticommute. Dualizing (i.e. applying $\Hom_R(-,R)$) and using the
isomorphism $\wedge ^{i}F^*\cong \wedge ^{t+c-1-i}F\otimes \wedge
^{t+c-1}F^*\cong \wedge ^{t+c-1-i}F$, it will be sufficient to prove the
commutativity of the following diagram
$$\begin{array}{cccccccccc}
G & \stackrel{\sigma _c^*}{\longleftarrow }&  G\otimes R \\
\uparrow \scriptstyle{\varphi}& &  \uparrow\scriptstyle{\epsilon ^*_0} \\
F    &  \stackrel{\sigma _{c-1}^*}{\longleftarrow } & G\otimes
  S_{0}G^*\otimes \wedge ^t F  \otimes \wedge^t G^* \\
  \, \hspace{1.2cm}\uparrow\scriptstyle{(-1)^{t+1}\epsilon_1^*} & &  \, \hspace{1.2cm} \uparrow   \scriptstyle{1\otimes (\partial _{c-1}^0)^*}\\
  \wedge ^{t+1}F\otimes \wedge ^tG^*\otimes S_0G^*
 & \stackrel{\sigma _{c-2}^*}{\longleftarrow } & G\otimes  S_{1}G^*\otimes \wedge ^{t+1}
   F  \otimes \wedge^t G^* \\
   \| & & & \\
   (S_0G\otimes \wedge^{c-2}F)^* & & &
\end{array} .
$$
We take $\{y_i\}_{i=1}^{t+c-1}$ to be a free $R$-basis of $F$ and $\{y_{i}^*
\}_{i=1}^{t+c-1}$ to be the dual basis, and we let $\{\, i_1,i_2,
\cdots,i_{t+1}\}$, $ i_1 < i_2< \cdots < i_{t+1}$ be a subset of
$I:=\{1,2,\cdots,t+c-1\}$. According to \cite{Ki} or \cite{eise}; pages 592 and 593 (see
also the exact sequence (\ref{splice})), $\epsilon _0^*$ and $\epsilon _1^*$
are defined by
$$\epsilon_0^*(g\otimes y_{i_1}\wedge y_{i_2}\wedge \cdots \wedge y_{i_t}):=
s_{t+1} \cdot g\otimes \varphi (y_{i_1})\wedge \cdots \wedge  \varphi (y_{i_t}), \text{ and }$$
$$\epsilon_1^*(y_{i_1}\wedge y_{i_2}\cdots \wedge y_{i_{t+1}}):=
\sum_{j=1}^{t+1} s_j \cdot (\varphi (y_{i_1})\wedge \cdots \wedge \varphi
(y_{i_{j-1}})\wedge \varphi (y_{i_{j+1}}) \wedge \cdots \varphi
(y_{i_{t+1}}))y_{i_j}.$$ where $s_j$ is the sign of the permutation of $I$
that takes the elements of $\{\, i_1, i_2, \cdots, i_{j-1},i_{j+1},\cdots,
i_{t+1}\}$ into the first $t$ positions. Note that $\varphi(y_i)$ are the
columns of the matrix $\cA$ associated to $\varphi $, that $\varphi
(y_{i_1})\wedge \cdots \wedge \varphi (y_{i_t})$ is the maximal minor
corresponding to the columns $i_1, \cdots , i_t$ and that a replacement of
$\epsilon_1^*$ by $(-1)^{t+1}\epsilon_1^*$, cf. the diagram above, still makes
the leftmost column in Diagram A a free resolution of $ S_{c-2}M$. We define:
$$\sigma _c^*=Id_G, \text{ and }$$
$$\sigma _{c-1}^*(g\otimes y_{i_1}\wedge \cdots \wedge y_{i_t})=\sum _{j=1}^t
(\varphi (y_{i_1})\wedge \cdots \wedge \varphi (y_{i_{j-1}})\wedge g\wedge
\varphi (y_{i_{j+1}}) \wedge \cdots \varphi (y_{i_t}))y_{i_j}.$$ A
straightforward computation gives us the desired commutativity, namely, $$
\sigma _c^*\cdot \epsilon_0^*=\varphi \cdot \sigma _{c-1}^* \quad \text{ and }
\quad (-1)^{t+1}\epsilon_1^* \cdot \sigma _{c-2}^*=\sigma _{c-1}^* \cdot
(1\otimes (\partial _{c-1}^0)^*).$$

We will now define the morphism $\ell _{\bullet }: Q_{\bullet }\longrightarrow F_{\bullet} [1]$. For any integer $i$, $0\le i \le c-2$, we  define:
$$\ell _{i}: S_{c-2-i}G \otimes \wedge ^{i}F \longrightarrow F^*\otimes S_{c-2-i}G\otimes \wedge ^{i+1}F$$
sending an element $m\otimes f \in S_{c-2-i}G\otimes \wedge ^{i}F$ to $\ell _{i}(m\otimes f)=\sum _{j=1}^{t+c-1}y_{j}^*\otimes m\otimes (y_j\wedge f)\in F^*\otimes S_{c-2-i}G\otimes \wedge ^{i+1}F$. Using  the diagram

\[
\xymatrix{
  &S_{c-4-i}G\otimes \wedge^{i+2}F  \ar[r]_{\sigma _{i+2} }
  \ar[d]_{\partial _{i+2}^{c-4-i}} &G^*\otimes S_{c-3-i}G\otimes \wedge^{i+2}F  \ar[r]^{\varphi^*\otimes 1 } \ar[d]_{1\otimes \partial _{i+2}^{c-3-i}} &F^*\otimes S_{c-3-i}G\otimes \wedge^{i+2} \ar[d]^{1\otimes \partial _{i+2}^{c-3-i}}\\
      &S_{c-3-i}G\otimes \wedge^{i+1}F   \ar[r]_{\sigma _{i+1}}
       \ar[d]_{\partial _{i+1}^{c-3-i}} \ar[rru]_{\ell _{i+1}} &G^*\otimes S_{c-2-i}G\otimes \wedge^{i+1}F   \ar[r]^{\varphi^*\otimes 1 } \ar[d]_{1\otimes \partial _{i+1}^{c-2-i}} &F^*\otimes S_{c-2-i}G\otimes \wedge^{i+1}F \ar[d]^{1\otimes \partial _{i+1}^{c-2-i}}\\
           &S_{c-2-i}G\otimes \wedge^{i}F \ar[r]_{\sigma _{i}}  \ar[rru]_{\ell _{i}}
  &G^*\otimes S_{c-1-i}G\otimes \wedge^{i}F
   \ar[r]^{\varphi^*\otimes 1 } &F^*\otimes S_{c-1-i}G\otimes
    \wedge^{i}F
 }\]

 \noindent and \cite{eise}; Proposition A2.8 page 583 onto the derivation $\varphi^*(x_j^*)$, we check for any $i$, $0\le i \le c-3$, that
 $$ (\varphi^* \otimes 1) \cdot \sigma _{i+1} = (1_{F^*}\otimes \partial ^{c-3-i}_{i+2}) \cdot \ell _{i+1} + \ell _{i} \cdot \partial _{i+1}^{c-3-i}.$$
We dualize the top part of Diagram A and we define $\ell _c^*$ and $\ell _{c-1}^*$ (obviously $\ell _i^* =0$ for $i\ge c+1$) as follows:
$$\ell _c^*=Id_F, \text{ and }$$
$$\ell _{c-1}^*(f\otimes y_{i_1}\wedge \cdots \wedge y_{i_t})=f\wedge y_{i_1}\wedge \cdots \wedge y_{i_t}.$$
A direct calculation using the following diagram
\[
\xymatrix{&0 &0 &0\\
  &G  \ar[u] &G\otimes R  \ar[u] \ar[l]_{\sigma_c^* } &F\otimes R \ar[u] \ar[l]_{\varphi\otimes 1}\ar[lld] ^{\ell_{c}^*}\\
      &F   \ar[u]_{\varphi}
      &G\otimes S_0G^*\otimes \wedge  ^tF\otimes \wedge^{t}G ^*   \ar[l]^{\sigma_{c-1}^* } \ar[u]^{1\otimes \epsilon _{0}^*} &F\otimes S_0G^*\otimes \wedge  ^tF\otimes \wedge^{t}G ^*   \ar[u]_{\epsilon _{0}^*}  \ar[l]_{\varphi\otimes 1} \ar[lld] ^{\ell_{c-1}^*}\\
           &\wedge ^{t+1}F\otimes \wedge ^tG^*\otimes    S_{0}G
           \ar[u]_{(-1)^{t+1}\epsilon _{1}^* }
  &G\otimes S_1G^*\otimes \wedge  ^{t+1}F\otimes \wedge^{t}G ^*
  \ar[l]^{\sigma_{c-1}^* } \ar[u]^{1\otimes (\partial _{c-1}^{0})^*} &F\otimes S_1G^*\otimes \wedge  ^{t+1}F\otimes \wedge^{t}G ^* \ar[l]^{\varphi\otimes 1} \ar[u]_{1\otimes (\partial _{c-1}^{0})^*}
 }\]

\noindent gives us (recall that $\sigma_c^* =1$):
 $$\varphi \cdot \ell _c^* =\sigma_c^* \cdot (\varphi\otimes 1), \text{ and}$$
 $$(-1)^{t+1}\epsilon _1^* \cdot \ell _{c-1}^* +\ell _{c}^*\cdot \epsilon _0^*=\sigma_{c-1}^* \cdot (\varphi\otimes 1).$$

Now we are ready to apply Lemma \ref{technicallemma}. Since  $\ell _c= Id$ and  $\sigma _{c}=Id$, the corresponding summands split off and we deduce that $\pd \Ext^1_R(M,S_{c-1}M)=c$, i.e. $\Ext^1_R(M,S_{c-1}M)$ is a Maximal Cohen-Macaulay $A$-module.

\vskip 2mm
Finally we prove the last assertion  of the Theorem. We assume
$a_j=1$ for all $j$ and $b_i=0$ for all $i$. Recall that the degree of a
linear standard determinantal scheme $X\subset \PP^n$ of codimension $c$ is
${t+c-1\choose c}$. Using the exact sequence (\ref{mainseq}), we get
$$  \ _{-2}\Ext^1(M,S_{c-1}M)=0, \text{ and }$$
$$ \dim _K(_{-1}\,\Ext^1(M,S_{c-1}M))=(t+c-1) \cdot {t+c-2\choose c-1}= \rk (\Ext^1(M,S_{c-1}M))\cdot \deg(X).$$

\noindent Therefore, $\Ext^1(M,S_{c-1}M)$ is a maximal Cohen-Macaulay module maximally generated or, equivalently, $\widetilde{\Ext^1_R(M,S_{c-1}M)}$ is an Ulrich sheaf on $X$ of rank $c$.
\end{proof}

Given an ACM scheme $X\subset \PP^N$ with dualizing sheaf $\omega $ and a coherent sheaf $\shE$ on $X$, we denote by $\shE^{\omega}$ the sheaf $\shH om_{{\mathcal O}_X}(\shE,\omega)$. It is well known that $\shE$ is ACM if and only if $\shE^{\omega}$ is ACM. So, as a Corollary of Theorem \ref{normalACM}, we have:

\begin{corollary}
 We keep the notation introduced above  and we set $N:=\Ext^1_R(M,S_{c-1}M)$. If $\depth _JA\ge 2$ then,   $\Hom_A(N,K_A)$  is a maximal Cohen-Macaulay
$A$-module of rank $c$. Moreover,  if $a_j=1$ for all $j$ and $b_i=0$ for all $i$, then $\widetilde{\Hom_A(N,K_A)}$ is an Ulrich sheaf of rank $c$.
\end{corollary}

\begin{remark}\label{Ext1(M,SiM)} \rm
 Keeping the notation introduced  above and arguing as in the proof of Theorem \ref{normalACM},  we can prove that for any $i$, $0\le i \le c$, there exits an exact sequence
  \begin{equation}\label{seqExt1(M,SiM)}
    0 \longrightarrow S_{i-1}M   \longrightarrow G^*\otimes S_{i}M  \longrightarrow F^*\otimes S_{i}M \longrightarrow \Ext^1
    _R(M, S_{i}M)  \longrightarrow 0 \end{equation}
  provided $\depth _JA\ge 2$  where $J=I_{t-1}(\cA)$. Note that we interpret
  $S_{-1}M$ as  $\Hom_A(M,A)$.

  Hence we get a big diagram, similar to diagram A above, where we have
  replaced $c-1$ by $i$. The part of the proof of Theorem \ref{normalACM}
  where we show the existence of $\{\sigma_j\}_{j=0}^c$ between the two
  leftmost columns (we call this part of the big diagram by (*)), seems to
  hold for any $i$, $0\le i \le c$. Since $\sigma_{i}$ is injective (except
  for $i=c$), this would imply that the length of the projective resolution of
  $\Ext^1 _R(M, S_{i}M)$ is at most $\pd A +1$ for $0\le i \le c-1$. The only
  problem to get this result will be to define $\sigma _j$ more generally and to verify the commutativity of the first
  diagram in the proof of Theorem \ref{normalACM} where the dual of the
  ``splice'' maps $\epsilon_j, j=0,1$ occur (we call this diagram (**) after
  having replaced $c-1$ by $i$ and made the corresponding obvious changes).

  The case $i=1$ was treated in \cite{K2011}; Proof of Theorem 3.1. That proof
  is almost the dual of the proof of this part of Theorem \ref{normalACM}.
  Indeed the lower part of diagram (*) for $i=1$ is exactly diagram (**)
  provided we move $G^*$ from the second column to $G$ in the first column in
  diagram (*). Thus the proof of Theorem 3.1 in \cite{K2011} implies this
  part of the proof of Theorem \ref{normalACM} and vice versa.

  The existence of the morphisms $\sigma_j$ in diagram (*) in the case $i=c$ is
  very similar (and easier) to what we had to prove in Theorem \ref{normalACM}
  for $i=c-1$. Indeed, in this case we only need to check diagram (**) where
  now only one $\epsilon_j$ occur. The mapping cone construction leads to a
  resolution of $\Ext^1 _R(M, S_{c}M)$ where, however, the leftmost free
  module (of rank one) clearly does not split off, whence the length of a
  minimal resolution must be $\pd A +2$. As explained for $i=1$ as almost the
  ``dual'' of $i=c-1$ above, the case $i=0$ is similarly ``dual'' to $i=c$. We
  deduce the existence of a morphism between the $R$-free resolutions of $
  S_{-1}M$ and $G^* \otimes A$. In this case it is easy to see that the
  leftmost free module in the resolution of $\Ext^1 _R(M, A)$ split off. In
  particular, we get
\[ \pd \Ext^1 _R(M, S_{i}M) \le \pd A +1 \ \ {\rm for} \ \ i \in \{0,1,c-1\} \
. \]
\end{remark}

The importance of proving $\pd \Ext^1 _R(M, S_{i}M)$ to be small follows from
our next proposition because it leads to good depth of ``twisted normal
modules''. Indeed, in \cite{K2011} we use this to prove both conjectures
appearing in \cite{KM11} on the dimension and smoothness of
the locus of
determinantal schemes inside the Hilbert scheme.
\begin{proposition} \label{ext-norm} With the above notation, set $J=I_{t-1}(\cA)$. We have
\begin{itemize}
\item[(1)] $\Ext^1_R(M,S_{c-1}M)\cong \Hom _A(I/I^2,S_{c-2}M)$
provided $\depth _J A\ge 2$; and
\item[(2)] $\Ext^1_R(M,S_{i}M)\cong \Hom _A(I/I^2,S_{i-1}M)$ for $0\le i \le c$, $i\ne c-1$,
provided $\depth _J A\ge 4$.
\end{itemize}
\end{proposition}
\begin{proof} We prove (1) and (2) simultaneously. By  Proposition \ref{resol} (ii), the canonical module $K_A(v)\cong S_{c-1}(M)$ for some integer $v$ and by Proposition \ref{resol} (ii) $M$ is a maximal Cohen-Macaulay $A$-module. Therefore, we have
\begin{equation}\label{extMK}
\Ext  ^{i}_A(M,K_A)=0 \text{ for } i\ge 1.
\end{equation}
Using \eqref{NM} as in the proof of Theorem 4.1 in \cite{K2011}, we get
that
\begin{equation}\label{ext(M,SiM)}
\Ext  ^{j}_A(M,S_iM)=0 \text{ for } 1\le j\le 2,
\end{equation}
provided $\depth _J A\ge 4$.

Under the assumption $\depth _J A\ge 2$, we have seen in (\ref{Hom(M,SiM)})
that for any $i$, $1 \le i \le c$
$$
\Hom_A(M,S_{i}M) \cong S_{i-1}M
$$
(true also for $i=0$). In particular, we have \begin{equation}\label{homMK}
  \Hom_A(M,K_A)\cong S_{c-2}(M)(-v).\end{equation}

Notice that the isomorphism $$\Hom_R(M,S_{i}M)\cong \Hom _A(M\otimes _R A,S_{i}M)$$ leads to a spectral sequence $$\Ext ^p_A(\Tor_q ^R(M,A),S_{i}M)\Rightarrow \Ext ^{p+q}_R(M,S_{i}M);$$
a part of the usual 5-term associated sequence is
$$0\longrightarrow \Ext ^{1}_A(M\otimes _R A,S_{i}M)\longrightarrow
\Ext ^{1}_R(M,S_{i}M)\longrightarrow \Hom_A(\Tor_1 ^R(M,A),S_{i}M)\longrightarrow \Ext ^{2}_A(M\otimes _R A,S_{i}M),
$$
which using (\ref{extMK}), (\ref{ext(M,SiM)})  and the exactness of the sequence
 $$0\longrightarrow \Tor_1 ^R(M,A)\longrightarrow M\otimes _R I \longrightarrow M\otimes _R R \stackrel {\simeq }
{ \longrightarrow} M\otimes _R A\longrightarrow 0$$
allows us to conclude that  $$\Ext ^{1}_R(M,S_{i}M)\cong \Hom_A(M\otimes _R I,S_{i}M).$$
 Using \eqref{NM} we get $$\begin{array}{ccl} \Hom_A (M\otimes _R I,S_{i}M) & \cong & \Hom_A (M\otimes I/I^2,S_{i}M) \\
  &\cong & \Hom _A(I/I^2, \Hom _A(M,S_{i}M)) \\ & \cong & \Hom _A(I/I^2,
  S_{i-1}M), \end{array}$$ and we are done.
 \end{proof}

 \begin{remark}\label{lastrem} \rm We can improve upon (2) of Proposition
   \ref{ext-norm} in the case $i=0$ and get
$$\Ext^1_R(M,A)\cong  \Hom _A(I/I^2, \Hom _A(M,A))$$
only assuming $\depth _JA\ge 3$. Indeed this depth condition implies
$\Ext^1_A(M,A)=0$ by \eqref{NM} and if we can show $\Ext^2_A(M,A)=0$ of
\eqref{ext(M,SiM)} by another argument, then the proof above applies to get
the claim. To see $\Ext^2_A(M,A)=0$ we remark that $S_{c-1}M$ is a twist of
the canonical module. This implies $\Ext^2_A(M,A)\cong \Ext^2_A(M \otimes
S_{c-1}M ,S_{c-1}M)$ by a spectral sequence argument, while we get $\Ext^2_A(M
\otimes S_{c-1}M ,S_{c-1}M) \cong \Ext^2_A(S_{c}M,S_{c-1}M)$ by using that $
\widetilde M \otimes \widetilde{S_{c-1}M} \cong \widetilde{S_{c}M}$ if we
restrict to $\Proj(A)-V(J)$, cf. \cite{K2011}; proof of Theorem 4.5 (the text
after the diagram (4.3)) for details. Then we conclude by Gorenstein duality.
\end{remark}

 \begin{theorem} \label{normalresolution} Let $X\subset \PP^n$ be a standard
   determinantal scheme of codimension $c\ge 2$ associated to a $t\times
   (t+c-1)$ matrix $\cA$. Set $J=I_{t-1}(\cA)$ and assume $\depth _JR/I(X)\ge
   2$. Then, ${\mathcal H}om _{\odi{X}}(\mathcal{I}_X/\mathcal{I}_X^2,
   \widetilde{S_{c-2}M})$ is an ACM sheaf of rank $c$. In addition, if $X$
   is a linear standard determinantal scheme, then ${\mathcal H}om
   _{\odi{X}}(\mathcal{I}_X/\mathcal{I}_X^2, \widetilde{S_{c-2}M})(-H)$ is an
   initialized Ulrich sheaf of rank $c$ and it has a pure linear
   $\odi{\PP^n}$-resolution of the following type:
 $$
0 \arr \odi{\PP^n}(-c)^{a_{c}}\arr \dots \arr \odi{\PP^n}(-1)^{a_1}\arr\odi{\PP^n}^{a_0}\arr {\mathcal H}om _{\odi{X}}(\mathcal{I}_X/\mathcal{I}_X^2,
   \widetilde{S_{c-2}M})(-H) \arr 0
$$
with $a_{0}=c\cdot \deg(X)=c\cdot {t+c-1\choose c}$ and $a_i={c\choose i}\cdot
a_0$ for $1\le i \le c$. Finally, if $X$ is a local complete intersection then
${\mathcal H}om _{\odi{X}}(\mathcal{I}_X/\mathcal{I}_X^2,
\widetilde{S_{c-2}M}) \cong {\mathcal N}_X\otimes \widetilde{S_{c-2}M}$ and
$$\Hl^{i}_*(X,{\mathcal N}_X\otimes \widetilde{S_{c-2}M})=0 \ \ {\rm for} \ 1\le
i\le n-c-1.$$
\end{theorem}
\begin{proof} 
We can apply Proposition
  \ref{ext-norm} and we get ${\mathcal H}om
  _{\odi{X}}(\mathcal{I}_X/\mathcal{I}_X^2, \widetilde{S_{c-2}M}) \cong
  \widetilde{\Ext^1_R(M,S_{c-1}M)}$. Therefore,
  applying Theorem \ref{normalACM} we conclude that ${\mathcal H}om
  _{\odi{X}}(\mathcal{I}_X/\mathcal{I}_X^2,
  \widetilde{S_{c-2}M})$ is an
  ACM sheaf of rank $c$ on $X$ and if $X$ is a linear standard determinantal
  scheme, then ${\mathcal H}om _{\odi{X}}(\mathcal{I}_X/\mathcal{I}_X^2,
  \widetilde{S_{c-2}M})(-H)$ is an initialized Ulrich sheaf of rank $c$, cf. the Corollary below
  for the twist. In
  this case the minimal $\odi{\PP^n}$-resolution of ${\mathcal H}om
  _{\odi{X}}(\mathcal{I}_X/\mathcal{I}_X^2, \widetilde{S_{c-2}M})(-H)$ is
  given by \cite{ESW}, Proposition 2.1.
\end{proof}

As an immediate application of Theorem \ref{normalACM} and its proof we obtain
a free $R$-resolution 
of $\Ext^1_R(M,S_{c-1}M)$, and more generally of $\Ext^1_R(M,S_{i}M)$ for
$0\le i \le c$, which we make explicit for later use. In fact, we have:

\begin{corollary}\label{res}
  We keep the notation introduced above and we call $Q_{\bullet}$,
  $P_{\bullet}$ and $F_{\bullet}$ the minimal free $R$-resolutions of
  $S_{i-1}M$, $G^*\otimes S_{i}M$ and $F^*\otimes S_{i}M$, respectively. Then
  $\Ext^1_R(M,S_{i}M)$, for $0\le i \le c$, has a free $R$-resolution of the
  following type:
$$
0\longrightarrow Q_{c}\longrightarrow Q_{c-1}\oplus P_{c} \longrightarrow Q_{c-2}\oplus P_{c-1}\oplus F_c \longrightarrow Q_{c-3}\oplus P_{c-2}\oplus F_{c-1} \longrightarrow \cdots $$
$$ \longrightarrow Q_{0}\oplus P_{1}\oplus F_{2} \longrightarrow  P_{0}\oplus F_{1} \longrightarrow F_{0} \longrightarrow \Ext^1_R(M,S_{i}M)\longrightarrow 0
$$
provided $\depth _JA\ge 2$ where $J=I_{t-1}(\cA)$. In particular, we have
$$\depth \Ext^1(M,S_{i}M)\ge \depth A-2.$$
Moreover if \, $i=c-1$ (resp. $ i \in \{0,1\}$) we may delete $Q_{c}, P_{c}$
and $Q_{c-1}, F_c$ (resp. $Q_{c}$ and a subsummand of $P_{c}$) from this
resolution, and we get $$\depth \Ext^1(M,S_{i}M)\ge \depth A-1 \text{ for } i=0,1; \text{ and }$$
$$\depth \Ext^1(M,S_{c-1}M)= \depth A.$$
\end{corollary}

\begin{proof} Using Lemma \ref{technicallemma} and Remark \ref{Ext1(M,SiM)} we
  get a free $R$-resolution of $\Ext^1_R(M,S_{i}M)$ of the form above. If
  $i=c-1$ (resp. $ i \in \{0,1\}$) the mentioned summands split off by the
  proof of Theorem \ref{normalACM} (resp. Remark \ref{Ext1(M,SiM)}).
\end{proof}
Note that we do not claim that this resolution is minimal (not even for $i =
c-1$) since we have not carefully analyzed all possible cancelation of
repeated direct summands in the mapping cone construction. Nevertheless, there is a particular case where we can assure that {\em all} repeated summands split off. Indeed, if $i=c-1$, $a_j=1$ for all $j$
and $b_s=0$ for all $s$ then  the module $\Ext^1_R(M,S_{c-1}M)$ is Ulrich.  By Proposition \ref{equivconditionsulrich} (ii), it has a pure linear resolution and therefore, we may delete not only $Q_{c}, P_{c}$ and $Q_{c-1}, F_c$ but also any other  repeated summand split off.

\vskip 2mm We would like to know whether the rank $c$ ACM (resp. Ulrich)
sheaves on $X$ constructed in the Theorem \ref{normalACM} are indecomposable.
We will see in the next section that under some weak conditions ${\mathcal
  N}_X$ and ${\mathcal H}om (\mathcal{I}_X/\mathcal{I}_X^2,
\widetilde{S_{i}M})$, $0\le i \le c-1$ are in fact indecomposable (see
Theorems \ref{simpleNormal} and \ref{indecomposableNormal}).

\vskip 2mm Now we consider and discuss a Conjecture which generalizes the main
result of this section. It is based on a series of examples computed with
Macaulay2 (\cite{Mac2}):

\begin{conj}\label{conjetura} Let $X\subset \PP^n$ be a standard determinantal scheme of
codimension $c\ge 2$ associated to a $t\times (t+c-1)$ matrix $\cA$. Set $I=I_t(\cA)$, $J=I_{t-1}(\cA)$ and assume $\depth _JR/I\ge 2$. With the above notation,  we conjecture that for all integer $i$, $0\le i \le c$,
$\Ext^{i}_R(S_{i}M,S_{c-i}M)$
is an (indecomposable) Maximal Cohen-Macaulay $R/I$-module of rank ${c\choose i}$.  In addition, if $X$ is a linear determinantal scheme then
$\Ext^{i}_R(S_{i}M,S_{c-i}M)$ is an (indecomposable) Ulrich $R/I$-module of rank ${c\choose i}$.
\end{conj}

Our next goal will be to  use the  spectral sequence
\begin{equation}\label{speqextSM}E_2^{p,q}=\Ext ^p_A(\Tor_q ^R(S_rM,A),S_{s}M)\Rightarrow \Ext
^{p+q}_R(S_rM,S_{s}M) \ \ \text{ for } -1 \le s\le c,
\end{equation}
 to get some non-obvious
isomorphic variations of the above Conjecture. Indeed under some natural $\depth _J$-conditions we will see that the spectral sequence degenerates and that we get 
\begin{equation}\label{ext(SrM,SsM)} \Ext^{a}_R(S_rM,S_sM)\cong \Hom_A(\wedge ^a(I/I^2),S_{s-r}M) \ \text{ for } \ 0\le a\le r-s+c.
\end{equation}

Firstly, we remark that since $T:=\Spec(A)\setminus V(J)$ is locally a complete intersection, then the ${\cO}_T$-modules
\begin{equation}\label{speqtorSM}
\widetilde{\Tor}_q ^R(\widetilde{S_rM},\cO _T)\cong \widetilde{\Tor}_q
^R(\cO _T,\cO _T)\otimes _ {\cO _T}\widetilde{S_rM} \ \ \text{ and }$$
$$\widetilde{\Tor}_q ^R(\cO _T,\cO _T)\cong \wedge ^q({\mathcal I }/{\mathcal I }^2)
\end{equation}
are locally free on $T$, whence if $\depth _J S_sM\ge i+2$ then (\ref{NM})
implies
\begin{equation}\label{SrM,SsM} \begin{array}{ccl}
    E_2^{p,q} = \Ext ^p_A(\Tor_q ^R(S_rM,A),S_{s}M) & \cong & \Hl^p(T,{\mathcal H}om(\wedge ^q({\mathcal I}/{\mathcal I}^2) \otimes \widetilde{S_rM},\widetilde{S_sM})) \\
    & \cong & \Hl^p(T,{\mathcal H}om(\wedge ^q({\mathcal I}/{\mathcal I}^2),\widetilde{S_{s-r}M})), \text{ for } p\le i.\end{array}
\end{equation}

\begin{lemma}\label{key} Let  $r$, $s$ and $a \ge 1$ be integers satisfying 
  $-1 \le s \le c$ and assume \begin{equation}\label{depthHom}
    \depth _J \Hom_A(\wedge ^q(I/I^2),S_{s-r}M)\ge a+3-q \ \text{ for all } \
    0\le q<a \, .
\end{equation}
Then for $j=r$, and more generally for every integer $j$  such that $r-s-1 \le j \le r-s+c$ we get
 $$ 
\Ext^{a}_R(S_jM,S_{s-r+j}M)\cong
\Hom_A(\wedge ^{a}(I/I^2),S_{s-r}M)\, .$$
\end{lemma}

\begin{remark}\rm The inequality
  (\ref{depthHom}) for $q=0$ just means $\depth _J S_{s-r}M\ge a+3$, whence
  $\depth _JA\ge a+3$. It follows that \eqref{SrM,SsM} holds for $p \le a+1$
  provided $-1 \le s \le c$.
\end{remark}

\begin{proof} Using \eqref{SrM,SsM} and the hypothesis (\ref{depthHom}) which
  yields
\begin{equation}\label{aux9} \Hl^{a+1-q}(T,{\mathcal H}om(\wedge ^q({\mathcal
    I}/{\mathcal I}^2),\widetilde{S_{s-r}M}))=\Hl_J^{a+2-q}(\Hom(\wedge ^q(I/
  I^2),S_{s-r}M))=0 \, ,\end{equation} we get  $E_2^{p,q}: =\Ext ^p_A(\Tor_q
^R(S_rM,A),S_{s}M)=0 \ \text{ for } \ p+q=a+1, \ 0\le q<a$. In the same way
\begin{equation*} \Hl^{a-q}(T,{\mathcal H}om(\wedge ^q({\mathcal I}/{\mathcal
    I}^2),\widetilde{S_{s-r}M}))=0 \end{equation*} implies that $E_2^{p,q} =0 \
\text{ for } \ p+q=a, \ 0\le q<a.$ Hence all terms $E_2^{p,q}$ with $p+q=a$
vanish except $E_2^{0,a}$and we get
$$\Ext^{a}_R(S_rM,S_sM)\cong E_{\infty}^{0,a}.$$
We claim that (\ref{aux9}) implies $E_{\infty}^{0,a}\cong E_2^{0,a}$. Indeed, by (\ref{aux9}),   $E_{\mu }^{p,q}=0$ for any $\mu \ge 2$ and $(p,q)$ satisfying $p+q=a+1$, $0\le q<a$.
Therefore, the differentials of the spectral sequence
$$d_{\mu, 1-\mu }:E^{0,a}_{\mu}\longrightarrow    E^{\mu,a+1-\mu}_{\mu} \ \ , \ \mu \ge 2 $$
vanish for $\mu \ge 2$ because $E^{\mu,a+1-\mu}_{\mu} =0$ for $\mu \ge 2$. It
follows that
\begin{equation}\label{aux101/2}
\Ext^{a}_R(S_rM,S_sM)\cong E_2^{0,a}\cong \Hom_A(\wedge ^{a}(I/I^2),S_{s-r}M)
\end{equation}
where the isomorphism to the right follows from $\depth _J S_{s-r}M \ge 2$.
Finally since the arguments above apply similarly to the spectral sequence
$'E_2^{p,q}: =\Ext ^p_A(\Tor_q ^R(S_jM,A),S_{s-r+j}M)$ as they did for
$E_2^{p,q}$ we are done.
 \end{proof}

 \begin{example} \rm Suppose $\depth _JA\ge 5$ and take $a=r=2$. Then
   (\ref{depthHom}) is satisfied for $s=2$; the case $q=0$ by hypothesis and
   the case $q=1$ follows from \cite{K2011}; Theorem 5.11. Therefore, we get
\begin{equation}\label{aux11} \Ext^2_R(S_2M,S_2M)\cong \Hom_A(\wedge
  ^2I/I^2,A).\end{equation}
If $s=c$ then (\ref{depthHom})
is satisfied; the case $q=1$ follows from
Theorem \ref{normalACM} and Proposition \ref{ext-norm}, whence
\begin{equation*} \Ext^2_R(S_2M,S_cM)\cong \Hom_A(\wedge
  ^2I/I^2,S_{c-2}M).\end{equation*}
\end{example}

\begin{lemma}\label{pd}  Let $r, s$ and $k$ be integers such that 
  $r-1\le s\le r-1+c$ and $1\le k \le \min \{r-s+c,c\}$, and suppose
  $\depth_JA\ge 2$. Moreover for every $i, j$ satisfying $1 \le i < k$, $i \le
  j \le k$ we assume that $$\Ext^{i}_R(S_kM,S_{s-r+k}M)\cong
  \Ext^{i}_R(S_jM,S_{s-r+j}M) \,.$$ Then we have 
\begin{equation}\label{pdExt} \pd \Ext^k_R(S_kM,S_{s-r+k}M)\le
    \pd A+2k \,. 
\end{equation}
\end{lemma}
\begin{proof} We proceed by induction on $k$. 
The case $k=1$ follows from
  Corollary \ref{res}. To prove it for $k > 1$, we will use that 
  $$\Ext^{i}_R(S_{i}M,S_{s-r+i}M)\cong \Ext^{i}_R(S_{k}M,S_{\sigma}M) \,, \ \   \sigma:=s-r+k \,,$$ for $i < k $ by assumption. By Proposition \ref{resol}, we have a resolution of length $c$ of $S_kM$ of the following type:
 $$\rightarrow L_{k+1} \stackrel{\epsilon }{\rightarrow }  L_k=\wedge ^kF \rightarrow L_{k-1}=G\otimes \wedge^{k-1}F \rightarrow \cdots \rightarrow L_1=S_{k-1}G\otimes F\rightarrow L_0=S_kG\rightarrow S_kM\rightarrow 0$$
 where $\epsilon $ is the "splice map". Applying the contravariant functor
 $\Hom_R(-,S_{\sigma}M)$, we get a complex
  \begin{equation}\label{complex}   0\rightarrow
    S_{\sigma}M\otimes L^*_0 \stackrel{\delta _0 }{\rightarrow }
    S_{\sigma}M\otimes L^*_1 \stackrel{\delta _1 }{\rightarrow }\cdots \rightarrow
    S_{\sigma}M\otimes L^*_{k-1} \stackrel{\delta _{k-1} }{\rightarrow }
    S_{\sigma}M\otimes L^*_{k} \stackrel{\epsilon ^*}{\rightarrow }
    S_{\sigma}M\otimes L^*_{k+1} \end{equation}
  where $\ker (\delta _0)=\Hom(S_kM,S_{\sigma}M)\cong S_{{\sigma}-k}M$, $\epsilon ^*=\Hom(\epsilon, S_{\sigma}M)=0$,
  whence         $$\Ext^k_R(S_kM,S_{\sigma}M)=S_{\sigma}M\otimes L^*_k/\im(\delta _{k-1}).$$
  Due to the horseshoe lemma and the exact sequence
  $$0\longrightarrow \im(\delta _{i-1}) \longrightarrow \ker(\delta _{i})
  \longrightarrow \Ext^{i}_R(S_kM,S_{\sigma}M)\longrightarrow 0$$
  we get the implication
  $$\pd \im(\delta _{i-1})\le \pd A+2i-1 \ \Rightarrow \  \pd \ker(\delta _i)\le \pd A+2i$$
  by the induction hypothesis. The exact sequence
   $$0\longrightarrow \ker(\delta _{i}) \longrightarrow S_{\sigma}M\otimes L_i^*
   \longrightarrow \im(\delta _{i})\longrightarrow 0$$   shows the implication
  $$\pd  \ker(\delta _{i})\le \pd A+2i \ \Rightarrow \  \pd \im(\delta _i)\le \pd
  A+2i+1 \ \text{ for } \ 1\le i < k.$$ Since by (\ref{complex}) and
  $\ker(\delta _0)\cong S_{{\sigma}-k}M$ we have
 $$\pd\im (\delta _0)\le pd A+1 \,,$$ we conclude that $\pd\im (\delta_{k-1})\le \pd A+2k-1$ and
 using the exact sequence
 $$0       \longrightarrow \im(\delta _{k-1})   \longrightarrow S_k(M)\otimes L_k^*\longrightarrow
 \Ext^k_R(S_kM,S_{\sigma}M)\longrightarrow 0$$
 we get    $$\pd \Ext^k_R(S_kM,S_{\sigma}M)\le \pd A+2k$$ which proves the lemma.
 \end{proof}
 \begin{remark} \label{r=s} \rm If $s \in \{r-1,r,r+c-2\}$, $s \le c$,
   we have $\pd \Ext^k_R(S_kM,S_{s-r+k}M) \le \pd A+1$ for $k=1$ by Remark
   \ref{Ext1(M,SiM)}. Using the proof above, we show $\pd\ker (\delta _1)\le
   pd A+1$, whence we can improve upon the conclusion for such $s$ and we get $$\pd
   \Ext^k_R(S_kM,S_{s-r+k}M)\le \pd A+2k-1  \ \ \text{ for } 1\le k \le \min \{r-s+c,c\}.$$
 \end{remark}

 \begin{theorem}\label{newthm} Let $X\subset \PP^n$ be a standard
   determinantal scheme of codimension $c\ge 2$ associated to a $t\times
   (t+c-1)$ matrix $\cA$. Let $I=I_{t}(\cA)$, $J=I_{t-1}(\cA)$ and $A=R/I$.
   Fix integers \ $s \ge -1$, $r-1\le s\le c$,
   and assume $\depth _JA\ge 2a+2$ (resp. $\depth _JA\ge 2a+1$ if $s \in
   \{r-1,r,r+c-2\}$ and $a > 1$). Then, for $0\le a\le r-s+c$, we have $$\pd_R
   \Ext^a_R(S_rM,S_sM)\le \pd _RA+2a \ , $$
$$\Ext^a_R(S_rM,S_sM)\cong \Hom_A(\wedge ^a(I/I^2),S_{s-r}M) \ , \ \text{ and }$$
$$  \Ext^a_R(S_jM,S_{s-r+j}M) \cong \Hom_A(\wedge ^a(I/I^2),S_{s-r}M)  \ {\rm for\ }
r-s-1 \le j \le \min \{r-s+c,c\} \,.$$
\end{theorem}
\begin{remark} \rm (i) Since we in general have $\depth _JA\le c+2$ and the
  assumption $\depth _JA\ge 2a+2$ (resp. $\depth _JA\ge 2a+1$) in Theorem
  \ref{newthm}, we get that $a$ satisfies $a \le \frac c 2$ (resp. $a \le \frac {c+1} 2$).

 {\rm (ii)} It is worthwhile to point out that Theorem \ref{newthm}
  with $r=s=0$ holds more generally in the licci case (see \cite{KP};
  Proposition 13 and Remark 14).
\end{remark}

\begin{proof} 
  Let $q$ be an integer such that $0\le q \le a$ where $ a \le r-s+c$ (whence
  $s-r \le c$). Then we will prove
\begin{equation} \label{depthExtSrSs} \depth\Ext^q_R(S_aM,S_{s-r+a}M)\ge \dim
  A-2q\ ,  \ \ {\rm  and}
\end{equation}
\begin{equation} \label{ExtqSrSs} \Ext^q_R(S_jM,S_{s-r+j}M) \cong \Hom
  _A(\wedge ^q(I/I^2),S_{s-r}M) \ {\rm for\ }
   r-s-1 \le j \le \min \{r-s+c,c\}\, 
\end{equation}
by induction on $q$. Note that \eqref{depthExtSrSs} is equivalent to $\pd \Ext^q_R(S_aM,S_{s-r+a}M) \le \pd A+2q $ by
Auslander-Buchsbaum's formula and that it implies $\depth
_J\Ext^q_R(S_aM,S_{s-r+a}M)\ge \depth _JA-2q $.


Since $S_{s-r}M$ is maximally Cohen-Macaulay, \eqref{depthExtSrSs} and
\eqref{ExtqSrSs} holds for $q=0$. Suppose both formulas hold for every
non-negative $q < k$ for some positive $k \le a$. It follows  that
\begin{equation} \label{depthHomq} \depth _J \Hom _A(\wedge
  ^q(I/I^2),S_{s-r}M)\ge \depth _JA-2q\ge 2k+2-2q\ge k+3-q;
\end{equation}
and applying Lemma \ref{key} (with $k$ instead of $a$) we get \eqref{ExtqSrSs}
for $q=k$. Since we now have \eqref{ExtqSrSs} for $q=k$ and $j \in \{k, a\}$
we get \eqref{depthExtSrSs} for $q=k$ by Lemma \ref{pd}, which completes the
induction. Moreover since we can take $j=r$ in \eqref{ExtqSrSs}, we get the
theorem for $s \notin \{r-1,r,r+c-2\}$.

Finally if $s \in \{r-1,r,r+c-2\}$ we use
Remark \ref{r=s} to improve upon \eqref{depthExtSrSs} and
\eqref{ExtqSrSs} for $q>0$. Using $\depth _JA\ge 2a+1$ we still get
\eqref{depthHomq}, and we are done.
\end{proof}

\begin{corollary} \label{equivconj} Let $s \ge -1$ and $a \ge0$ be integer and
  suppose \ $r-1 \le s \le c$ and $\depth _JA\ge 2a+2$ (resp. $\depth _JA\ge
  2a+1$ if $s \in \{r-1,r,r+c-2\}$ and $a > 1$). Then we have
$$  \Ext^i_R(S_rM,S_sM) \cong  \Ext^i_R(A,S_{s-r}M) \ \ \ {\rm for\ every \ }
   i \le  \min \{a,r-s+c\} \,.$$
 In particular, for
$i\le \min \{a, \frac {c+1} 2\}$, we have
$$\Ext^{i}_R(S_iM,S_{c-i}M)\cong \Ext^{i}_R(A,S_{c-2i}M)\cong\Ext^{i}_R(S_eM,S_{c-2i+e}M) \ \ \text{ for } \ \ 0\le e
\le 2i.$$
\end{corollary}

\begin{proof} Straightforward to verify using the two last formulas of
  Theorem~\ref{newthm}.
\end{proof}

\begin{remark} \label{correctionC14} \rm If we take a general determinantal
  scheme $X=\Proj(A)$ one knows that $\depth _JA = c+2$. Comparing this with
  $\depth _JA \ge 2a+2$ of Corollary~\ref{equivconj} we get the isomorphisms
  of the \ $\Ext^i_R$-groups of the corollary under the assumption $ i \le
  \frac c 2$ (if $s-r \le \frac c 2$). Thus Corollary~\ref{equivconj} corrects
  with a complete proof the previous version on the arXiv, as well as the
  published version in Crelle's journal, whose corresponding result
  (Proposition 3.9) only assumes $ i \le c$ (after renaming letters in
  Proposition 3.9). The problem with the proof there is that a morphism
  between the spectral sequences in the proof of Proposition 3.9 is not
  established. The proof of the Lemma's \ref{key}, \ref{pd} and Theorem
  \ref{newthm} are mainly the same as those in Crelle's journal, and their
  assumptions imply that the spectral sequences degenerate on the $E_2$ level
  (see \eqref{aux101/2}). Note that the assumptions in Lemma~\ref{key}
  (resp.\,Theorem \ref{newthm}) are exactly (resp.\,added $s \ge -1$) as in
  previous versions/Crelle'journal, while we in Lemma \ref{pd} have introduced
  an assumption on certain $\Ext^i_R$-groups which we obviously can replace by
  the assumptions in Theorem \ref{newthm} if $k \le a$. We only know
  counterexamples to Proposition 3.9 when $i = c$, and in fact we expect
  Corollary~\ref{equivconj} to be true for $i \le c-1$ and $X$ general without
  assuming $\depth _JA \ge 2a+2$ (resp. $\ge 2a+1$), see Remark
  \ref{correctConj} for a follow-up. The inaccuracy in Proposition 3.9 implies
  that we had to weaken Remark \ref{lastrem} (Remark 3.10 in previous
  versions) to show this result. Moreover an assumption of generality is
  included in Theorem~\ref{indecomposableNormal}.
\end{remark}

\begin{corollary} \label{vanis} Let $i \ge 1$, $s \ge -1$ and $a$ be integers
  satisfying \ $0\le a \le c -s$.
\begin{itemize}
\item[(i)]
 Suppose $\depth _JA\ge 2a+2+i$ (resp. $\depth _JA \ge
  2a+1+i$ if $s
\in \{-1,0,c-2\}$ and $a > 0$). Then, for $ 1\le j \le i$, we have
$$\Ext^{j}_A(\wedge^{a}(I/I^2),S_{s}M)=0 \  \text{ and } \ \Hl^j_*(X \setminus V(J),(\wedge^{a} \shI _X/\shI _X ^2)^{\vee} \otimes \widetilde{S_{s}M})=0.$$
\item[(ii)]Suppose $\dim A\ge 2a+2+i$ (resp. $\ge 2a+1+i$ if  $s
\in \{-1,0,c-2\}$ and $a > 0$).
  Then we have
$$ \Hl^j_*(X,\shH om(\wedge^{a} (\shI _X/\shI _X ^2),\widetilde{S_{s}M}))=0  \ \ \ {\rm for} \ \ 1\le j \le i .$$
\end{itemize}
\end{corollary}

\begin{proof} (i) Taking $r=0$ in Theorem \ref{newthm}, we get
$$
\Ext^a_R(A,S_{s}M)\cong \Hom _A(\wedge ^a(I/I^2),S_{s}M)$$ and applying Lemma
\ref{pd} together with Auslander-Buchsbaum's formula, as in
\eqref{depthExtSrSs}, we have
\begin{equation} \label{23}
\depth _J \Hom_A(\wedge ^a(I/I^2),S_{s}M)\ge 2+i.
\end{equation}
Since $\wedge ^{a}(\shI _X/\shI _X^2)$ is locally free on $X\setminus
V(J)$, we apply \eqref{NM} and, for $1\le j \le i$, we get
$$\Ext^{j}_A(\wedge^{a}(I/I^2),S_{s}M)=0 \ \text{ and } \ \Hl^j_*(X\setminus
V(J),\shH om(\wedge^{a} (\shI _X/\shI _X ^2),\widetilde{S_{s}M}))=0.$$ Suppose
$s \in \{-1,0,c-2\}$. Then, for $a =1$, we may use Remark \ref{r=s} to improve
upon \eqref{depthExtSrSs} by $1$ and since we still get \eqref{23}, we
conclude by the arguments above. If $a \ge 2$ then Theorem \ref{newthm}
applies, and using Remark \ref{r=s} to improve upon \eqref{depthExtSrSs}, we
get \eqref{23} and we conclude as previously.

(ii) The argument using the projective dimension and Auslander-Buchsbaum's
formula also imply \eqref{23} with $\mathfrak{m}$ instead of $J$, whence we
get $\Hl^{j+1}_{\mathfrak{m}}(\Hom_A(\wedge ^{a}(I/I^2),S_{s}M))=0$ and we are
done.
\end{proof}
The case $a=1$, $i=1$ and $s=-1$ of Corollary \ref{vanis} is of special
interest because it shows the vanishing of a group that we tried very much to
compute in \cite{KM05}. Indeed, we have
\begin{corollary} \label{vaniss} Let $J=I_{t-1}(\cA)$ and $A=R/I$ and suppose
  $\depth _JA\ge 4$. Then
$$\Ext^{1}_A(I/I^2,\Hom(M,A))=0.$$
\end{corollary}
In Theorem 5.1 of \cite{KM05} we repeatedly need to apply Corollary
\ref{vaniss} to standard determinantal schemes obtained by deleting at least
one column. Indeed, Corollary \ref{vaniss} shows that all assumptions of
Theorem 5.1 (ii) are satisfied in the case $\dim X \ge 2$, $a_j \ge b_i$ for
any $i,j$. It follows that the closure of the locus of determinantal schemes
$W$ inside the Hilbert scheme, is a generically smooth component of the Hilbert
scheme (if $W \ne \emptyset$)! Thus Corollary \ref{vaniss} partially reproves
Corollary 5.9 of \cite{K2011}. In particular, we also get Conjecture 4.2 of
\cite{KM11} from Theorem 5.1 of \cite{KM05} and Corollary \ref{vaniss}.

Since the case $a=1$ and $s=c-1$ is related to the dual of the conormal module
of $A$, we remark

\begin{corollary} \label{vanis2} Let $J=I_{t-1}(\cA)$, $A=R/I$, $X=\Proj(A)$
  and let $K_A$ be the canonical module of $A$. If $\depth _JA\ge 4+i $, then
  $$\Hl^{k}_{\mathfrak{m}}(I/I^2)=0 \ \ \ {\rm for } \ \ \dim X-i < k \le \dim X,$$ or
  equivalently ;
$$\Ext^{j}_A(I/I^2,K_A)=0 \ \ \ {\rm for} \ \ 1\le j \le i .$$
\end{corollary}
\begin{proof} This is immediate from Corollary \ref{vanis} (i) since
  $K_A(v)\cong S_{c-1}(M)$ for some integer $v$ by Proposition \ref{resol}
  (ii). For the equivalent statement, we use Gorenstein duality.
\end{proof}

\begin{remark} \label{correctConj} \rm Here we try to extend
  Corollary~\ref{equivconj} to 
the range $\frac c 2 \le i \le c-1$, as well as to show
\begin{equation}  \label{simE}
\Ext^i_R(S_rM,S_sM)\cong \Hom_A(\wedge ^i(I/I^2)\otimes S_{r}M,S_{s}M) \ \ {\rm when} \ \ -1\le s \le c 
\end{equation}
for such $i$. Note that since $S_{s}M$ is maximally Cohen-Macaulay
as an $A$-module, 
the $\Hom$-group has $\depth_J \ge 2$, and it is further isomorphic to $ \Hom_A(\wedge ^i(I/I^2),S_{s-r}M)$ if \ $ -1\le s-r \le c$.

{\rm (i)} Conversely if $N:=\Ext^i_R(S_rM,S_sM)$ satisfies $\depth_JN \ge 2$,
then \eqref{simE} holds. In fact if $T:=\Spec(A)\setminus V(J)$, the depth
assumption implies $N \cong H^0(T, \widetilde N)$. Then \eqref{simE} follows
easily from the local version of the spectral sequence \eqref{speqextSM} which
degenerates due to \eqref{speqtorSM}.

{\rm (ii)} One may express \eqref{simE} using $\Tor^R_{c-i}$-groups. Indeed
note that the resolutions  of $S_kM$ and $S_{c-1-k}M$ given in
Proposition \ref{resol} are $R$-dual to each other (up to twist by $ \ell :=\sum
_{j=1}^{t+c-1}a_j-\sum _{i=1}^tb_i$). Combining
with the fact that the homology groups of \eqref{complex} may be interpreted
as $\Ext^i_R(S_kM,S_{\sigma}M)$ as well as $\Tor^R_{c-i}(L,S_{\sigma}M)$ where
$L$ is determined by the $R$-dual resolution of $S_kM$, we get: 
$$\Ext^i_R(S_kM,S_{\sigma}M) \cong
\Tor^R_{c-i}(S_{c-1-k}M,S_{\sigma}M)(\ell) \ \ {\rm for} \ \ -1 \le
k,\sigma \le c \ .$$  
Indeed  letting $j=c-i$ and $h = c-1-r$ we have $$\Ext^i_R(S_rM,S_sM) \cong \Ext^i_R(S_{r+1}M,S_{s+1}M) \Longleftrightarrow \Tor^R_{j}(S_{h}M,S_{s}M) \cong \Tor^R_{j}(S_{h-1}M,S_{s+1}M) \,.$$
In particular
Corollary~\ref{equivconj} leads to a corresponding result for
$\Tor^R_j$-groups where $\frac c 2 \le j \le c$.

{\rm (iii)} Using {\rm (ii)} and $\Tor^R_{j}(L_1,L_2) \cong
\Tor^R_{j}(L_2,L_1)$ we get for every $i$, $0 \le i \le c$, the isomorphism
\begin{equation} \label{simEc} \Ext^i_R(S_rM,S_sM)\cong
  \Ext^i_R(S_{c-1-s}M,S_{c-1-r}M)\ \ \ {\rm for} \ \ -1 \le r, s \le c .
\end{equation}
{\rm (iv)} Exactly at the spot where the splice map in the complex
\eqref{complex} occurs, we do not only have $\Ext^k_R(S_kM,S_{\sigma}M)$ as a
certain cokernel, but also $\Ext^{k+1}_R(S_kM,S_{\sigma}M)$ as the kernel of
$S_{\sigma}M\otimes L^*_{k+1} \to S_{\sigma}M\otimes L^*_{k+2}$. Since
$H^0_J(-)$ is left exact and $ \depth_J S_{\sigma}M \ge 2$, we get $\depth_J
\Ext^{k+1}_R(S_kM,S_{\sigma}M) \ge 2$ for $k < c$ (for  $k = c$, $\Ext^{k+1}_R(S_kM,S_{\sigma}M)=0$), thus \eqref{simE} holds for
$i=r+1$, $ -1 \le r, s \le c$ by {\rm (i)}. Using also {\rm (ii)} it follows that both groups of \eqref{simEc} for $i=r+1$ 
 are  further isomorphic to 
$$ \Tor^R_{j}(S_{j}M,S_{s}M)(\ell) \cong \Hom_A(\wedge ^{r+1}(I/I^2) \otimes S_{r}M,S_{s}M) \,, \ j=c-1-r \ {\rm for} \ -1 \le r, s \le c \,.$$

{\rm (v)} So the bottom line for \eqref{simE} to hold is to show that
$\depth_J \Ext^i_R(S_rM,S_sM)\ge 2$. Let us point out one more interesting
case where this holds, namely the case $i=c-1= r-s$ (i.e. $(r,s)= (c,1)$,
$(c-1,0)$ or $(c-2,-1)$). Using {\rm (ii)} we have $\Ext_R^{c-1}(S_rM,S_sM)
\cong \Tor^R_1(S_hM,S_sM)(\ell)$ with $h+s=0$. Letting $(h,s)=(1,-1)$ or
$(-1,1)$ we get two $\Tor_1$-groups that obviously are isomorphic and {\rm
  (iv)} applies to get $ \Tor^R_{1}(S_{h}M,S_{s}M)(\ell) \cong \Hom_A(\wedge
^{c-1}(I/I^2) \otimes S_{c-1}M,A)$. If $(h,s)=(0,0)$, the $\Tor_1$-group is
isomorphic to $I/I^2$, the conormal module. Its depth is found in the next
section (Proposition~\ref{depthconormal}), and the assumptions $n \ge 2c$, $A$
general and $a_j>b_i$  suffice for having $\depth_J I/I^2 \ge
2$. Thus for $r-s =c-1$, $n \ge 2c$, $A$ general and $a_j>b_i$ for
all $i,j$,  we conclude that
$$ \Ext^{c-1}_R(S_rM,S_sM) \cong \Hom_A(\wedge ^{c-1}(I/I^2)\otimes S_{c-1}M,A) \cong I/I^2(\ell)\ .$$
\end{remark}

As another application of Theorem \ref{newthm},
we can partially restate Conjecture \ref{conjetura}. Indeed,

\begin{conj} \label{conjetura2}
 Let $X\subset \PP^n$ be a linear standard determinantal scheme of
codimension $c\ge 2$ associated to a $t\times (t+c-1)$ matrix $\cA$. Set
$I=I_t(\cA)$ and $J=I_{t-1}(\cA)$. For any integer $a$, $0 \le a\le
\frac{c+1}{2}$ we conjecture that
$\Hom(\wedge ^{a}(I/I^2),S_{c-2a}M)$ is an (indecomposable) Ulrich
$R/I$-module of rank ${c\choose a}$ provided  $\depth _JA\ge 2a+2$.
\end{conj}

Notice that under the above hypothesis Conjectures \ref{conjetura} and
\ref{conjetura2} are equivalent since, by Theorem \ref{newthm}, we have
$$\Ext^a_R(S_aM,S_{c-a}M)\cong\Hom_A(\wedge ^a(I/I^2),S_{c-2a}M) \ \text{ for } 1\le a \le \frac{c+1}{2}.$$

We want to point out that for $a=0$, the conjecture was proved in
\cite{BHU}; Proposition 2.8 and in this section, we prove
it for $a=1$ (and, indecomposability for $n\ge 2c+1$ in  the next section).

\section{The indecomposability of the normal sheaf of a determinantal variety}

In this section we address Problem \ref{pblm}(3) and we determine conditions
under which the normal sheaf $ {\mathcal N}_X$, and more generally the
 ``twisted'' normal sheaves ${\mathcal N}_X(\mathcal M^i):= {\mathcal H}om
_{\odi{X}}(\mathcal{I}_X/\mathcal{I}_X^2, \widetilde{S_{i}M})$, $-1 \le i \le
c-1$, of a standard determinantal scheme $X\subset \PP^n$ of codimension $c$ are
simple (i.e. $\Hom( {\mathcal N}_X(\mathcal M^i),{\mathcal N}_X(\mathcal
M^i))\cong K$) and, hence, indecomposable (cf. Theorems \ref{simpleNormal} and
\ref{indecomposableNormal}). Obviously ${\mathcal N}_X(\mathcal M^0)=
{\mathcal N}_X$. We keep the notation of previous sections. So, $\cA $ will be
the $t \times (t+c-1)$ matrix associated to the standard determinantal scheme
$X\subset \PP^n$, $I=I_t(\cA)$, $A=R/I$ and
\begin{equation} \label{eagonnorthcott}
\cdots \longrightarrow \oplus _jR(-n_{2j})\longrightarrow \oplus _{i}R(-n_{1i})
\longrightarrow R\longrightarrow A\longrightarrow 0
\end{equation}
the Eagon-Northcott resolution of $A$.  In this section we often have
$d_{i,j}:=a_j-b_i > 0$ for all $i,j$.

Again we will assume $c>1$, since the case $c=1$ corresponds to a hypersurface
$X\subset \PP^n$ and ${\mathcal N}_X\cong \odi{X}(\delta )$, $\delta
:=\deg(X)$, is simple. We will also assume $t>1$, since the case $t=1$
corresponds to a codimension $c$ complete intersection $X\subset \PP^n$ and
${\mathcal N}_X(\mathcal M^i)$ and ${\mathcal N}_X=\oplus _{i=1}^{c}
\odi{X}(d_i)$, $d_i\in \ZZ$, are neither simple nor indecomposable. Let us
start computing the depth of the conormal bundle of a standard determinantal
scheme. We have:

\begin{proposition} \label{depthconormal}
Let $\cA $ be a $t\times (t+c-1) $ homogeneous matrix with entries that are general forms of
degree $d_{ij}  > 0$, let $I=I_t(\cA)$, $J=I_{t-1}(\cA )$ and $X\subset \PP^n$ be the standard determinantal scheme of codimension $c$
associated to $\cA$. Assume that $n\ge 2c-2$. Then, it holds:
\begin{itemize}
\item[(1)] $\depth  _{\mathfrak{m}} I/I^2=n-2c+2.$
\item[(2)] $\depth _J I/I^2\ge 3$ (resp. $=n-2c+2$) provided $n\ge 2c+1$ (resp. $n\le 2c$).
\end{itemize}
 Therefore, if $n\ge 2c$ there is  a
   closed subset $Z\subset X$ such that $X\setminus Z\hookrightarrow \PP^n$ is a local
   complete intersection and
   $\depth_{I(Z)}I/I^2\ge 2$.
 \end{proposition}
 \begin{proof} (1)
We will first prove Proposition~\ref{depthconormal} for an ideal generated by the maximal minors of a matrix with entries that are
indeterminates and we will deduce that it also works for a
 homogeneous matrix with entries of general homogeneous
polynomials.

So, we first assume that  $n=t(t+c-1)-1$  and $\cA$ is a matrix with entries
$x_0,x_1, \cdots,x_n$ of indeterminates. Set $R=k[x_0,\cdots ,x_n]$, $I=I_t(\cA )$ and call $\Omega $
the module of differentials of $R/I$ over $k$. By \cite{b-v}; Theorem 14.12
$$   \depth  _{\mathfrak{m}} \Omega=\depth  _{\mathfrak{m}} R/I-\depth _J R/I+2.$$  and note that $\depth _J R/I= \codim (I_{t-1}(\cA),R/I)$. Since, $$ \depth  _{\mathfrak{m}} R/I=\dim R/I=n-c+1, \text{ and} $$
$$\codim  (I_{t-1}(\cA),R/I)=\codim (I_{t-1}(\cA))-\codim (I_t(\cA))=2(c+1)-c=c+2,$$
we get $$\depth  _{\mathfrak{m}} \Omega=(n-c+1)-(c+2)+2=t(t+c-1)-2c.$$
Therefore, using  the exact sequence
$$ 0\longrightarrow   I/I^{(2)} \longrightarrow (R/I)^{t(t+c-1)} \longrightarrow \Omega
\longrightarrow 0,$$  we deduce that
$$ \begin{array}{ccl} \depth  _{\mathfrak{m}} I/I^{(2)} & = &
\depth  _{\mathfrak{m}} \Omega +1  \\
 & = & t(t+c-1)-2c+1 \\
 & = & n-2c+2. \end{array}$$
 By \cite{b-v};  Corollary 9.18, we have $I^2=I^{(2)}$. Therefore, we conclude that
 $$ \depth  _{\mathfrak{m}} I/I^2=n-2c+2.$$

 Let us now assume $n<t(t+c-1)-1.$ We distinguish two cases:
\vskip 2mm
\noindent \underline{Case 1.} Assume $d_{ij}=1$ for all $i,j$ (i.e. the entries of
the matrix $\cA$ are linear forms).    We have  a  $t\times (t+c-1)$ matrix
$\overline{ \cA }=(x_{i,j})$ of indeterminates and the ideal
$I _{t}(\overline{ \cA })\subset S:=k[x_{i,j}]$ which verifies
$$\depth  _{\mathfrak{m}_S}I_t(\overline{ \cA})/I_t(\overline{ \cA })^2=t(t+c-1)-2c+1.$$
We choose $t(t+c-1)-n-1$ general linear forms $\ell _1, \cdots ,
\ell_{t(t+c-1)-n-1}\in S=k[x_{i,j}]$ and we set $S/(\ell _1,\cdots ,\ell
_{t(t+c-1)-n-1})\cong k[x_0,x_1, \cdots ,x_n]=:R$. Let us call $I\subset R$
the ideal of $R$ isomorphic to the ideal $I_{t}(\overline{ \cA })/(\ell
_1,\cdots ,\ell _{t(t+c-1)-n-1})$ of $S/(\ell _1,\cdots ,\ell
_{t(t+c-1)-n-1})$. $I$ is nothing but the ideal $I_{t}(\cA)$ where
$\cA=(m_{i,j})$ is a $t\times (t+c-1)$ homogeneous matrix with entries that
are linear forms in $k[x_0, x_1, \cdots , x_n]$ obtained from $\overline{ \cA
}=(x_{i,j})$ by substituting using the equations $\ell _1,\cdots ,\ell
_{t(t+c-1)-n-1}$. Since $\depth _{\mathfrak{m}_S} I_t(\overline{ \cA
})/I_t(\overline{ \cA })^2=t(t+c-1)-2c+1$ and $t(t+c-1)-n-1\le t(t+c-1)-2c+1$,
we can assume that $\ell _1,\cdots ,\ell _{t(t+c-1)-n-1}$ is a regular
sequence on both $I_t(\overline{ \cA })/I_t(\overline{ \cA })^2$ and
$S/I_t(\overline{ \cA })$; and we conclude that
$$\depth  _{\mathfrak{m}_R}I/I^2=\depth  _{\mathfrak{m}_S}I_t(\overline{ \cA })/I_t(\overline{ \cA
})^2-(t(t+c-1)-n-1)=n-2c+2.$$

\noindent \underline{Case 2.} Since $d_{i,j}> 0$ for all $i,j$, it is
enough to raise the entry $m_{i,j}$ of the above matrix $\cA$ to the power
$d_{i,j}$.

 (2) It follows from (1) and $\depth_J I/I^2\ge \depth _J A-(\dim A-\depth  _{\mathfrak{m}}I/I^2)$, cf. \cite{K2007}; Lemma 7 because $\depth _J A=c+2$ (resp. $n+1-c$) for $n\ge 2c+1$ (resp. $n\le 2c$).
\end{proof}

As an immediate and nice consequence of the above result, we get for $n\ge 2c+1$ that the cohomology of the conormal bundle $\Hl^{j}_{\mathfrak{m}}(I/I^2)$ is non-zero for only
  one value of $j \ne \dim R/I$, cf. \cite{A} for $c = 2$. Analogous result for the normal bundle was proved by Kleppe in \cite{K2011}; Theorem 5.11.

\begin{corollary} \label{vanis3} Let $\cA $ be a $t\times (t+c-1) $
  homogeneous matrix with entries that are general forms of positive
  degree. Set $J=I_{t-1}(\cA)$, $A=R/I$
  and let $K_A$ be the canonical module of $A$. We have:
  \begin{itemize}
  \item[(i)]  $\Hl^{k}_{\mathfrak{m}}(I/I^2)=0$ for $k < n-2c+2$ or,
  equivalently,
$\Ext^{j}_A(I/I^2,K_A)=0$ for $c\le j \le n-c+1.$
\item[(ii)]  $\Hl^{k}_{\mathfrak{m}}(I/I^2)=0$ for $max(3,n-2c+2)<k<n-c+1$.
\end{itemize}
\end{corollary}
\begin{proof} (i) It follows from Proposition \ref{depthconormal} and
  Gorenstein duality.

(ii) Since $\depth _JA=min(c+2,n-c+1)$, we can apply Corollary \ref{vanis2} and we get $\Hl^{k}_{\mathfrak{m}}(I/I^2)=0$ for $n-2c+2<k\le \dim X=n-c.$
Combining with (i) we get what we want.
\end{proof}

\begin{theorem} \label{simpleNormal} Let $X\subset \PP^n$ be a standard
  determinantal scheme of codimension $c\ge 2$ associated to a $t\times
  (t+c-1)$ matrix $\cA$ with entries that are general forms of positive
  degree. Let ${\mathcal N}_X(\mathcal M^i):= {\mathcal H}om
  _{\odi{X}}(\mathcal{I}_X/\mathcal{I}_X^2, \widetilde{S_{i}M})$ for $-1 \le i
  \le c-1$ and assume $n\ge 2c$. Then,
 $$\Hom( {\mathcal N}_X(\mathcal M^i),{\mathcal N}_X(\mathcal M^i))\cong \
 _0\!\Hom_A(I/I^2,I/I^2)\cong K$$ provided $max\{n_{2j}\}<2\cdot
 min\{n_{1i}\}$. In particular, ${\mathcal N}_X$ and ${\mathcal N}_X(\mathcal
 M^i)$ are simple, and thus indecomposable.
\end{theorem}

\begin{proof}  First of all we observe that $\Hom_R(I,I)\cong R$. Indeed, we have the following diagram:
 \[
 \begin{array}{cccccccccc}
 &  & & & 0 & & & \\ & & & & \downarrow & & & & & \\ & &  & &\Hom_R(R,R)=R & & & \\
  & & & & \downarrow & & & & & \\ 0 &   \longrightarrow
  &  \Hom_R(I,I)&  \longrightarrow & \Hom_R(I,R) &
 \longrightarrow & \Hom_R(I,A)  \longrightarrow & \Ext^1_R(I,I)
 \\ &  &&  &
 \ \ \downarrow & &  & &  & \\  & &
 & & \Ext^1_R(A,R)=0 & &    &
 \end{array}
 \]

 \noindent Since $c\ge 2$,  $\Ext^1_R(A,R)=0$ and $R\cong \Hom_R(R,R)\cong \Hom_R(I,R)$.
  Hence, it follows that $\Hom_R(I,I)$
 is an ideal of $\Hom_R(I,R)=R$ containing the identity; and so $\Hom_R(I,I)\cong R$.

\vskip 2mm
 \noindent {\bf Claim:}  If  $max\{n_{2j}\}<2\cdot min\{n_{1i}\}$, then
 $K\cong \  _0\!\Hom_R(I,I)\cong  \
   _0\!\Hom_A(I/I^2,I/I^2)$.

\vskip 2mm
\noindent {\bf Proof of the Claim:} We apply $\Hom_R(-,I)$ and
$\Hom_R(-,I/I^2)$ to the minimal resolution of $I$ deduced from
(\ref{eagonnorthcott}) and we get the following commutative diagram with exact
horizontal sequences

 \[
 \begin{array}{cccccccccc}
 &  & & &   0 &  & 0 & \\
  & & & &  \uparrow & & \uparrow & &  \\
 0 & \longrightarrow
 & \Hom_R(I,I/I^2)  & \longrightarrow  & \oplus_iI/I^2(n_{1i})& \longrightarrow &
 \oplus_j I/I^2(n_{2j})& \longrightarrow& \\
  & & \uparrow & & \uparrow & &  \uparrow & &  \\ 0 &   \longrightarrow
 & \Hom_R(I,I)  & \longrightarrow  & \oplus_iI(n_{1i})& \longrightarrow &
 \oplus_j I(n_{2j})&  \longrightarrow \\
   & &  & & \uparrow & & \uparrow & &  \\
  &  & &   & \oplus_i I^2(n_{1i}) & \longrightarrow  & \oplus_j I^2(n_{2j}) & \\
   & &  &  & \uparrow & & \uparrow & &  \\
   &  & & &   0 &  & 0 & \\
  \end{array}
 \]

\noindent Since $\Hom_R(I,I/I^2)\cong \Hom_A(I/I^2,I/I^2)$ and $_0\!\Hom_R(I,I^2)=0$, it suffices to show that
 $( I^2(n_{2j}))_0=0$ for all $j$. Using the natural surjective map $S_2I\twoheadrightarrow I^2$, it suffices to show that $(S_2I)_\mu=0$ for $\mu:=max\{n_{2j}\}$. But   $(S_2I)_\mu=0$ because we have a surjective map $\oplus _{i\le j}R(-n_{1i}-n_{1j})_{\mu}
  \cong S_2(\oplus _{i}R(-n_{1i}))_ {\mu}\twoheadrightarrow (S_2I)_\mu $ and
  $\oplus _{i\le j}R(-n_{1i}-n_{1j})_{\mu}=0$
  by the assumption  $max\{n_{2j}\}<2\cdot min\{n_{1i}\}$. Hence, the claim is proved.

\vskip 2mm
 Let us now prove that  ${\mathcal N}_X(\mathcal M^i)$ is simple, i.e.  $\Hom({\mathcal N}_X(\mathcal M^i),{\mathcal N}_X(\mathcal M^i))\cong \
 _0\!\Hom_A(I/I^2,I/I^2)\cong K$.  Set $J=I_{t-1}(\cA)$. Since $\depth _J R/I=min(c+2, n+1-c)\ge 2$, we get that $\depth _{J}N_M
 \ge 2$ where $N_M:=\Hom_A(I/I^2,S_iM)$. This also implies $\depth_{J}\Hom_A(N_M,N_M)\ge 2$,
 whence $$\Hom_A(N_M,N_M)
 \cong \Hl^0_*(X\setminus Z,{\mathcal H}om_{\odi{X}}({\mathcal N}_X(\mathcal M^i),{\mathcal N}_X(\mathcal M^i)))
 \cong \Hl^0_*(X,{\mathcal H}om_{\odi{X}}({\mathcal N}_X(\mathcal M^i),{\mathcal N}_X(\mathcal M^i)))
 $$
 where $Z:=V(J)$. Using that $ {\mathcal N}_X(\mathcal M^i) \cong  ({\mathcal
   I}_X/{\mathcal I}_X ^2)^{\vee } \otimes \widetilde {S_iM}$ is locally free on
 $X\setminus Z$, we have
 $$ \Hl^0_*(X\setminus Z, {\mathcal H}om _{\odi{X}}({\mathcal N}_X(\mathcal M^i),{\mathcal N}_X(\mathcal M^i)))=  \Hl^0_*(X\setminus Z, {\mathcal H}om _{\odi{X}} ({\mathcal I}_X/{\mathcal I}_X ^2,
            {\mathcal I}_X
 /{\mathcal I}_X ^2)).$$
 Since $n\ge 2c$, we can apply Proposition \ref{depthconormal} and we get
 $\depth_{I(Z)} I/I^2\ge 2$  which implies that $\depth _{I(Z)}\Hom_A(I/I^2,I/I^2)
 \ge 2$,
 whence $$\Hom_A(I/I^2,I/I^2)
 \cong \Hl^0_*(X\setminus Z,{\mathcal H}om _{\odi{X}}({\mathcal I}_X/{\mathcal I}_X ^2,{\mathcal I}_X
 /{\mathcal I}_X ^2)). $$
Putting altogether we obtain
             $$\Hl^0_*(X,{\mathcal H}om_{\odi{X}}({\mathcal N}_X(\mathcal M^i),{\mathcal N}_X(\mathcal M^i)))\cong
              \Hom_A(I/I^2,I/I^2);$$
and taking the degree zero piece of these graded modules we get what we want.
\end{proof}

\begin{remark} \label{simpleNormal2} \rm
In Theorem \ref{simpleNormal}, the hypothesis $n\ge 2c$ together with the generality of the entries
 can be replaced by
\begin{equation}\label{hypothesis1} \Hl^1_{I(Z)} (I/I^2)_{n_{1i}}=\Hl^0_{I(Z)} (I/I^2)_{n_{1i}}=\Hl^0_{I(Z)} (I/I^2)_{n_{2j}} =0 \ \ \text{ for any } i \text{ and } j
\end{equation}
where $Z\subset X$ is a closed subset such that $X\setminus Z\hookrightarrow \PP^n$ is a local complete intersection. Note that $I(Z)= \mathfrak{m}$ if we can take $Z=\emptyset $, e.g. if $n\le 2c+1$ and the entries are general forms. To show it, observe that the exact cohomology sequence associated to $0\rightarrow I^2\rightarrow I\rightarrow I/I^2\rightarrow 0$ gives us ($Z:=V(J)$)
$$\Hl^1_{I(Z)} (I/I^2)_{\mu}\cong \Hl^2_{I(Z)} (I^2)_{\mu } \cong \Hl^1(X\setminus Z,{\mathcal I}_X^2(\mu )) \ \ \text{ and} $$
$$\Hl^0_{I(Z)} (I/I^2)_{n_{2i}}\cong \Hl^1_{I(Z)} (I^2)_{n_{2i} } \cong \Hl^0(X\setminus Z,{\mathcal I}_X^2(n_{2i} )) \ \text{ since } \ I^2(n_{2i})_0=0.$$
In the proof of Theorem \ref{simpleNormal}, we show
 $$\begin{array}{ccl}\Hom_A(N_M,N_M)
   \cong \Hl^0_*(X\setminus Z, {\mathcal H}om _{\odi{X}} ({\mathcal I}_X/{\mathcal I}_X ^2,
   {\mathcal I}_X
   /{\mathcal I}_X ^2))\end{array}$$
 and we use  the hypothesis $n\ge 2c$ to get  $\depth_{I(Z)} I/I^2\ge 2$ and, hence,
\begin{equation}\label{aux3}\Hl^0_*(X\setminus Z,{\mathcal H}om _{\odi{X}}({\mathcal I}_X/{\mathcal I}_X ^2,{\mathcal I}_X
 /{\mathcal I}_X ^2))\cong
              \Hom_A(I/I^2,I/I^2).\end{equation}
However, to get the Theorem \ref{simpleNormal} we only need  the isomorphism (\ref{aux3}) in degree 0. Letting $\Ext^{j}_{I(Z)}(N,-)$ be the right derived functor of the
composed functor   $\Hl _{I(Z)}^0 \circ \Hom_R(N,-)$, cf. \cite{SGA2}, Expos\'{e} VI for details,
we have in degree 0 an exact sequence
$$0 \longrightarrow \  _0\! \Hom_{I(Z)}(I/I^2,I/I^2)\longrightarrow \ _0\!\Hom_A(I/I^2,I/I^2)\longrightarrow $$
$$ \Hl^0(X\setminus Z,{\mathcal H}om _{\odi{X}}({\mathcal I}_X/{\mathcal I}_X ^2,{\mathcal I}_X
 /{\mathcal I}_X ^2)) \longrightarrow \ _0\! \Ext^1_{I(Z)}(I/I^2,I/I^2)\longrightarrow $$ where
$\ _0\! \Hom_{I(Z)}(I/I^2,I/I^2)\cong \ _0\! \Hom(I/I^2, \Hl^0_{I(Z)} (I/I^2))$ and the terms in
$$
0 \longrightarrow \  _0\! \Ext^1(I/I^2, \Hl^0_{I(Z)} (I/I^2))\longrightarrow \ _0\! \Ext^1_{I(Z)}(I/I^2,I/I^2)\longrightarrow \ _0\! \Hom(I/I^2, \Hl^1_{I(Z)} (I/I^2))
$$ vanish by the assumptions (\ref{hypothesis1}). Hence, we conclude $$\Hl^0(X\setminus Z,{\mathcal H}om _{\odi{X}}({\mathcal I}_X/{\mathcal I}_X ^2,{\mathcal I}_X
 /{\mathcal I}_X ^2))\cong \
              _0\!\Hom_A(I/I^2,I/I^2)$$ which proves that in Theorem \ref{simpleNormal}
              we can replace the hypothesis $n\ge 2c$  by
the assumptions (\ref{hypothesis1}).
\end{remark}

\begin{remark} \rm   If $X$ is a linear standard determinantal scheme and $t>1$ then
$n_{1i}=t$ for all $i$, $n_{2j}=t+1$ for all $j$, and the hypothesis $max\{n_{2j}\}<2\cdot min\{n_{1i}\}$ is satisfied.
\end{remark}

\begin{remark}\label{rem} \rm It is worthwhile to point out that in Theorem \ref{simpleNormal}   the  hypothesis $max\{n_{2j}\}<2\cdot min\{n_{1i}\}$ cannot be dropped when $c=2$. To prove it we will compute
 the cokernel of  the morphism $K=R_0\cong \ _0 \! \Hom_R(I,I) \stackrel{\phi _0 }{\longrightarrow } \ _0 \! \Hom_A(I/I^2,I/I^2)$. To this end, we consider  the diagram

{\small $$\begin{array}{ccccccccc}\Hom_R(I,I) & \stackrel{\phi  }{\longrightarrow } & \Hom_R(I,I/I^2) & \rightarrow  & \Ext^1_R(I,I^2)  & \rightarrow & \Ext^1_R(I,I) &  \longrightarrow & \Ext^1_R(I,I/I^2) \\
 & & & & \downarrow \cong &  & \downarrow  \cong & & \downarrow \cong  \\
 & & & & \Ext^1_R(I,R)\otimes I^2 & \rightarrow & K_A (n+1)\otimes _RI &
 \stackrel{\cong  }{\longrightarrow } & K_A (n+1)\otimes _RI/I^2 . \\ & & & & \downarrow \cong \\ & & & & K_A (n+1)\otimes _R I^2
 \end{array} $$}
where we have used $\Ext^2_R(I,-)=0$ to get the isomorphisms. Therefore, $$\coker(\phi )=  K_A (n+1)\otimes _R I^2 \cong K_A (n+1)\otimes _A I^2/I^3.$$

 On the other hand for a standard determinantal scheme $X\subset \PP^n$ of codimension 2, we have the following well known exact sequences:
 $$0 \longrightarrow  G^*=\bigoplus _{j=1}^tR(-n_{2j})\longrightarrow F^*=\bigoplus _{i=1}^{t+1}R(-n_{1i}) \longrightarrow I\longrightarrow 0, \text{ and }$$
  \begin{equation} \label{newseq} 0 \longrightarrow  \wedge ^2G^*\longrightarrow G^*\otimes F^*  \longrightarrow S_2F^*\longrightarrow S_2I\longrightarrow 0.
  \end{equation}
   Since a generic complete intersection of codimension 2 is syzygetic, i.e. $S_2I\cong I^2$, the exact sequence
  $$ \cdots \longrightarrow F \longrightarrow G \longrightarrow K_A(n+1) \longrightarrow 0$$
 leads to a commutative diagram
  $$\begin{array}{ccccccc}  & &  G\otimes _R G^*\otimes _R F^* \\
   &  & \downarrow   & &   \\
   F\otimes _R S_2F^* & \longrightarrow  & G\otimes _R S_2F^* &   &   &  &  \\
  \downarrow  &  & \downarrow   & &   \\
  F\otimes _R I^2 & \longrightarrow & G\otimes _R I^2 & \longrightarrow  & K_A (n+1)\otimes _R I^2  &  \longrightarrow & 0.
 \end{array}. $$

\noindent It follows that {\small $$(F\otimes S_2F^*)\oplus (G\otimes G^*\otimes F^*) \longrightarrow G\otimes S_2F^*
=\bigoplus _{1\le i \le j\le t+1 \atop 1\le k\le t }R(-n_{1i}-n_{1j}+n_{2k})\rightarrow K_A(n+1)\otimes _R I^2 \rightarrow 0$$} is exact
and $\phi _0$ is not surjective e.g. in the case ($n_{21}=n_{22}=4$ and $n_{11}=2<n_{12}=n_{13}=3$):
$$ 0 \longrightarrow  R(-4)^2  \longrightarrow  R(-3)^2\oplus R(-2) \longrightarrow  I  \longrightarrow  0.$$
\end{remark}

\vskip 2mm
We have stated Theorem \ref{simpleNormal}  for standard determinantal  schemes because  the paper concerns the main features of the normal sheaf of a standard determinantal  scheme $X\subset \PP^n$. Nevertheless, the result works in a much more general set up. Indeed, the assumption  that $X$ is a standard determinantal scheme is not necessary and the result and its proof hold for any graded quotient of the  polynomial ring provided  $\depth_{I(Z)}I/I^2\ge 2$. In fact, it holds

\begin{theorem}\label{simple}
 Let $X\subset \PP^n$ be a closed subscheme of
 codimension $c\ge 2$ (not necessarily ACM)
 with a minimal free $R$-resolution $$
  \cdots \longrightarrow \oplus _j^{b_2} R(-n_{2j})\longrightarrow \oplus _{i}^{b_1}R(-n_{1i})
\longrightarrow R\longrightarrow R/I \longrightarrow 0
$$
where $I:=I(X)$. Let $Z\subset X$ be a closed subset such that $X\setminus
Z\hookrightarrow \PP^n$ is a local complete intersection. Let
$L$ be a finitely generated $R/I$-module that is invertible over $X\setminus
Z$, put ${\mathcal N}_X(\mathcal L):= {\mathcal H}om
_{\odi{X}}(\mathcal{I}_X/\mathcal{I}_X^2, \widetilde{L})$ and assume $\depth
_{I(Z)}L\ge 2$, $\depth_{I(Z)}I/I^2\ge 2$ and $max\{n_{2j}\}<2\cdot
min\{n_{1i}\}$. Then,
 $$\Hom( {\mathcal N}_X(\mathcal L),{\mathcal N}_X(\mathcal L))\cong \
 _0\!\Hom_A(I/I^2,I/I^2)\cong K.$$
\end{theorem}

\begin{remark} \rm For a complete intersection of codimension $c\ge 2$ and dimension $n-c\ge 1$, the conclusion is false while all assumptions, except for $max\{n_{2j}\}<2\cdot min\{n_{1i}\}$ are obviously satisfied.
\end{remark}

\begin{remark} \rm
As explained in Remark \ref{simpleNormal2},
 the hypothesis $\depth_{I(Z)}I/I^2\ge 2$ in Theorem \ref{simple} can be replaced by
\begin{equation}\label{hypothesis2} \Hl^1_{I(Z)} (I/I^2)_{n_{1i}}=\Hl^0_{I(Z)} (I/I^2)_{n_{1i}}=\Hl^0_{I(Z)} (I/I^2)_{n_{2j}} =0  \ \text{ for any } i \text{ and } j.
\end{equation}
\end{remark}

As an application we have:

\begin{corollary}  Let $X\subset \PP^n$ be either a codimension 2 ACM subscheme with a minimal free $R$-resolution $$
 0\longrightarrow \oplus _j^{\nu} R(-n_{2j})\longrightarrow \oplus _{i}^{\nu +1}R(-n_{1i})
\longrightarrow R\longrightarrow R/I \longrightarrow 0
$$ or a codimension 3 arithmetically Gorenstein subscheme
with a minimal free $R$-resolution $$
 0\longrightarrow R(-e) \longrightarrow \oplus _j^{\nu}R(-n_{2j})\longrightarrow \oplus _{i}^{\nu}R(-n_{1i})
\longrightarrow R\longrightarrow R/I \longrightarrow 0.
$$
Let $Z\subset X$ be a closed subset such that $X\setminus Z\hookrightarrow \PP^n$ is a local complete intersection. Assume $\depth _{I(Z)}R/I\ge \begin{cases} 3 & \text{ if } c=2 \\ 2   & \text{ if } c=3 \end{cases}$ and \ $max\{n_{2j}\}<2\cdot min\{n_{1i}\}$. Then,
 ${\mathcal N}_X$ is simple.
\end{corollary}
\begin{proof} By Theorem \ref{simple} it suffices to show that $\depth _{I(Z)} I/I^2 \ge 2$. For $c=2$, we get $\depth _{I(Z)} I/I^2 \ge \depth _{I(Z)} R/I -1$ by \cite{A}, cf. (\ref{newseq}). For $c=3$ we know that $I/I^2$ is MCM by \cite{Bu}
 because $I/I^2\otimes K_A$ is  MCM in the licci case and the canonical module $K_A$ is trivial in the Gorenstein case. So, we are done
\end{proof}

Let us give some examples of ACM schemes $X\subset \PP^N$ with simple normal sheaf.

\begin{example} \rm Let $X\subset \PP^3$ be a (smooth) ACM curve. Since
  $\depth  I/I^2 \ge \dim R/I-1$, we have $\Hl^0_{\mathfrak{m}} (I/I^2)_{\mu }=0$ for any $\mu$. So, according to Theorem \ref{simpleNormal} and Remark \ref{simpleNormal2}, we only need to check
  \begin{equation}\label{assumption} \Hl^1_{\mathfrak{m}}
    (I/I^2)_{n_{1i}}\cong \Hl^1(X, {\mathcal I}_X^2(n_{1i}))=0 \text{ for any }
    i; \text{ and }
\end{equation}
$$max\{n_{2j} \}<2 \cdot min\{n_{1i}\}$$
to conclude that $\mathcal{N}_X$ is simple. Using Macaulay2, we get
\begin{itemize}
\item[(i)] (\ref{assumption}) does not hold  in the linear case.
  \vskip 2mm
  \item[(ii)]  If $\deg {\mathcal A}=\begin{pmatrix} 1 & 1 & 2 \\ 1 & 1 & 2\end{pmatrix} $, then (\ref{assumption}) holds. However, since  the condition
    $4= max\{n_{2j} \}<2\cdot min\{n_{1i}\}=4$  is not true,  we cannot conclude that $\mathcal{N}_X$ is simple. In fact, we have seen in Remark \ref{rem} that ${\mathcal N}_X$ is not simple.
    \vskip 2mm
    \item[(ii)]  If $\deg {\mathcal A}=\begin{pmatrix} 1 & 2 & 2 \\ 1 & 2 & 2\end{pmatrix} $ or $\begin{pmatrix} 2 & 2 & 2 \\ 2 & 2 & 2\end{pmatrix} $ or $\begin{pmatrix} 3 & 2 & 2 \\ 3 & 2 & 2\end{pmatrix} $ or $\begin{pmatrix} 3 & 3 & 1 \\ 3 & 3 & 1\end{pmatrix} $, then (\ref{assumption}) holds as well as  the inequality
    $max\{n_{2j} \}<2 \cdot min\{n_{1i}\}$; and we get that $\mathcal{N}_X$ is simple.
\end{itemize}
In conclusion, the assumptions (\ref{assumption}) and $max\{n_{2j} \}<2\cdot min\{n_{1i}\}$ seem  weak for ACM curves in $\PP^3$.

\end{example}

\begin{example} \rm
Let $X\subset \PP^4$ be a smooth standard determinantal  curve. By Theorem \ref{simpleNormal} and Remark \ref{simpleNormal2}, to prove that $\mathcal{N}_X$ is simple, we only need to check that the following hypothesis are satisfied:
\begin{equation}\label{assumption1} \Hl^1_{\mathfrak{m}} (I/I^2)_{n_{1i}}=\Hl^0_{\mathfrak{m}} (I/I^2)_{n_{1i}}=\Hl^0_{\mathfrak{m}} (I/I^2)_{n_{2j}} =0 \text{ for any } i \text{ and } j; \text{ and }
\end{equation}
$$max\{n_{2j} \}<2\cdot min\{n_{1i}\}.$$
 Using Macaulay2, we get
\begin{itemize}
\item[(i)] (\ref{assumption1}) does not hold  in the linear case.
  \vskip 2mm
  \item[(ii)]  If $\deg {\mathcal A}=\begin{pmatrix} 1 & 2 & 2 & 2 \\ 1 & 2 & 2 & 2\end{pmatrix} $, then (\ref{assumption1}) holds while    $max\{n_{2j} \}<2\cdot min\{n_{1i}\}$  is not true and we cannot conclude that $\mathcal{N}_X$ is simple .
    \vskip 2mm
    \item[(ii)]  If $\deg {\mathcal A}=\begin{pmatrix} 2 & 2 & 2 & 2 \\ 2 & 2 & 2 & 2\end{pmatrix} $, then (\ref{assumption1}) holds as well as  the inequality
    $max\{n_{2j} \}<2 \cdot min\{n_{1i}\}$. Hence,  $\mathcal{N}_X$ is simple.
\end{itemize}
\end{example}

We will end this section with a result about the indecomposability of the normal sheaf of a standard determinantal scheme $X\subset \PP^n$ which does not involve the degrees of the generators (resp. first syzygies) of $I(X)$. To achieve our goal we need the following preliminary lemma.

\begin{lemma}\label{auxiliar}  Let $X\subset \PP^n$ be a codimension $c$ standard determinantal subscheme  associated to a  graded morphism $\varphi :F\longrightarrow  G$.   Set $M=\coker (\varphi)$ and $J=I_{t-1}(\varphi )$. Let $H\cong \PP^{n-1}\subset \PP^n$ be a hyperplane defined by a general  form $h\in R_1$.
Consider $X'=X\cap H\subset H$ the hyperplane section of $X$ and $R'=R/(h)\cong K[z_0, \cdots ,z_{n-1}]$. Then $X'$ is a codimension $c$ standard determinantal subscheme of $\PP^{n-1}$ associated to the  graded morphism $\varphi'=\varphi \otimes 1:F':=F\otimes _R R/(h)\longrightarrow  G':=G\otimes _ R R/(h)$
and  $M':=\coker(\varphi ')\cong M\otimes _R R/(h)$.
Assume $\depth _JA\ge 3$. Then, we have
 $$\Ext^1_R(M,S_{i}M)\otimes _R R/(h)\cong \Ext^1_{R'}(M',S_{i}M') \text{ for } 0\le i \le c.$$
\end{lemma}
\begin{proof} Let us consider the free $R$-resolution $W_{\bullet}$ of
  $\Ext^1_R(M,S_{i}M)$ of Corollary \ref{res}:
$$
0\longrightarrow Q_{c}\longrightarrow Q_{c-1}\oplus P_{c} \longrightarrow Q_{c-2}\oplus P_{c-1}\oplus F_c \longrightarrow Q_{c-3}\oplus P_{c-2}\oplus F_{c-1} \longrightarrow \cdots $$
$$ \longrightarrow Q_{0}\oplus P_{1}\oplus F_{2} \longrightarrow  P_{0}\oplus
F_{1} \longrightarrow F_{0} \longrightarrow \Ext^1_R(M,S_{i}M)\longrightarrow
0\, ;
$$
and the free $R'$-resolution $W'_{\bullet}$ of $\Ext^1_{R'}(M',S_{i}M')$:
$$
0\longrightarrow Q'_{c}\longrightarrow Q'_{c-1}\oplus P'_{c} \longrightarrow Q'_{c-2}\oplus P'_{c-1}\oplus F'_c \longrightarrow Q'_{c-3}\oplus P'_{c-2}\oplus F'_{c-1} \longrightarrow \cdots $$
$$ \longrightarrow Q'_{0}\oplus P'_{1}\oplus F'_{2} \longrightarrow
P'_{0}\oplus F'_{1} \longrightarrow F'_{0} \longrightarrow
\Ext^1_{R'}(M',S_{i}M')\longrightarrow 0
$$
also given by Corollary \ref{res} because $\depth _JA\ge 3$. Since $h\in R_1$ is
a general linear form and $\depth _{\mathfrak{m}} \Ext^1_R(M,S_{i}M)\ge 1$ by
Corollary \ref{res}, we have $$\Ext^1_R(M,S_{i}M) :h=\Ext^1_R(M,S_{i}M)$$ and
by \cite{bh}; Lemma 1.3.5, that $W_{\bullet } \otimes _RR/(h)=W_{\bullet }'$
is a free $R'=R/(h)$-resolution of $\Ext^1_R(M,S_{i}M)\otimes _RR/(h)$.
Therefore, we conclude that $$\Ext^1_R(M,S_{i}M)\otimes _R R/(h)\cong
\Ext^1_{R'}(M',S_{i}M').$$
\end{proof}
Now, we are ready to prove the indecomposability of the normal sheaf of a
standard determinantal scheme $X\subset \PP^n$ under some mild hypothesis
which does not involve the degrees of the generators (resp. first syzygies) of
$I(X)$. Recalling ${\mathcal N}_X(\mathcal M^k):= {\mathcal H}om
_{\odi{X}}(\mathcal{I}_X/\mathcal{I}_X^2, \widetilde{S_{k}M})$, we have

\begin{theorem}\label{indecomposableNormal} Let $X\subset \PP^n$ be a
  standard determinantal scheme of codimension $c\ge 2$ defined by a matrix
  $\cA$ with entries that are general forms. We keep the notation introduced
  above and set $J=I_{t-1}(\cA )$. Assume $a_j-b_i>0$ for all $i,j$ and
$$\depth _JA\ge \begin{cases} c+2 & \text{ if } c>2 \,,  \\ 3 & \text{ if } c=2 \,.  \end{cases}
  $$
  If $c=2$ and $n=4$, or $c\ge 2$ and $n\ge 2c+1$, then the normal sheaf
  ${\mathcal N}_X$ and more generally, the ``twisted'' normal sheaves
  ${\mathcal N}_X(\mathcal M^k)$, $-1 \le k \le c-1$, are indecomposable.
\end{theorem}
\begin{proof} The idea is to fix $c$ and use induction on $n$.  In fact in the special case $(c,n)=(2,4)$, $X$ is smooth because $\cA$ is general 
  and ${\mathcal N}_X$ is indecomposable by \cite{BP}; Th\'{e}or\`{e}me A. If  $(c,n)= (c,2c+1)$, $c\ge 2$, then $X$ is again smooth, $\Pic(X)\cong \ZZ^2$ (cf. Theorem \ref{picard}) and ${\mathcal N}_X$ is
  indecomposable by \cite{B}; Theorem 3.2. It follows that $ {\mathcal
    N}_X(\mathcal M^k) \cong ({\mathcal I}_X/{\mathcal I}_X ^2)^{\vee }
  \otimes \widetilde {S_kM}$ is indecomposable. The result now follows from
  induction using Lemma \ref{auxiliar} and taking into account that
  $\Ext^1_R(M,S_{k+1}M)\cong \Hom _A(I/I^2,S_kM)$ (see Proposition
  \ref{ext-norm}).
\end{proof}

\begin{example}\label{moreexamples} \rm (1) We consider $X\subset \PP^n$, $n\ge 4$,  the standard determinantal subscheme of codimension 2 associated to the matrix $\begin{pmatrix} x_0 & x_1 & x_2^2 \\ x_3 & x_4 & f \end{pmatrix} $ where $f$ is a general form of degree
2. We have seen in Remark \ref{rem} that ${\mathcal N}_X$ is not simple but it follows from Theorem \ref{indecomposableNormal} that ${\mathcal N}_X$ is indecomposable.

(2)  We consider a {\em rational normal scroll}
 $S(a_0, \ldots,a_k)$; i.e.  the image of the map
$$\sigma : \PP^1 \times \PP^k \longrightarrow \PP^N$$ given by
$$\sigma (x,y;t_0,t_1\cdots ,t_k):=(x^{a_0}t_0,x^{a_0 -1}yt_0, \cdots,y^{a_0}t_0,
,\cdots ,x^{a_k}t_k,x^{a_k -1}yt_k, \cdots,y^{a_k}t_k)$$
where $N=k+\sum _{i=0}^ka_i$. If
we choose coordinates $X_0^0,\cdots ,X_{a_{0}}^0, \cdots ,X_0^k,\cdots ,X_{a_{k}}^k$ in $\PP^N$, the ideal of $S(a_0, \ldots,a_k)$ is generated by the maximal minors of the $2\times c$ matrix with two rows and $k+1$ catalecticant blocks: $$M_{a_0, \cdots, a_k}:=\begin{pmatrix} X_0^0 & \cdots & X_{a_{0}-1 }^0 & \cdots &  X_0^k & \cdots & X_{a_{k}-1 }^k \\
X_1^0 & \cdots & X_{a_{0}}^0 & \cdots &  X_1^k & \cdots & X_{a_{k} }^k\end{pmatrix} .$$
By Theorem \ref{indecomposableNormal}, the normal bundle of
$S(1, \cdots , 1)$ and of $S(2, 1, \cdots  1)$ are indecomposable. (Notice that $S(1, \cdots , 1)$ corresponds to the Segre variety $\PP^1 \times \PP^k \hookrightarrow \PP^{2k+1}$ already discussed in \cite{B}; Corollary 3.3).
\end{example}

Arguing as in the proof of Theorem  \ref{indecomposableNormal}, we get

\begin{corollary} Let $X\subset \PP^n$ be a smooth standard determinantal scheme of codimension $c$ and let $\PP^{n-1}\cong H\subset \PP^n$ a general hyperplane and set
 $X'=X\cap H$. Assume $n-c\ge 2$ and $\depth _JA\ge 3$, $J=I_{t-1}(\cA)$. Then,  $({\mathcal N}_{X/\PP^n})_{|\PP^{n-1}}\cong {\mathcal N}_{X'/\PP^{n-1}}.$
\end{corollary}


\section{The $\mu$-semistability of the normal sheaf of a determinantal varieties}

 The normal bundle of a smooth variety $X \subset \PP^n$ has been intensively studied since it reflects many properties of the embedding; so far few examples of smooth varieties having $\mu$-(semi)stable normal bundle are known. The first example of a curve $C\subset \PP ^3$ with normal bundle ${\mathcal N}_{C}$ $\mu $-stable was given by Sacchiero (\cite{S}). In \cite{El}; Proposition 2, Ellia proved that the normal bundle of a linear determinantal curve $C\subset \PP^3$ is linear $\mu$-semistable and the  normal bundle ${\mathcal N}_{C}$ of a general ACM curve $C\subset \PP^3$ of degree 6 and genus 3 is $\mu$-stable (see  \cite{EH} and \cite{P} for more information about the $\mu$-stability of the normal sheaf of a curve in $\PP^3$). The first goal of this last section is to generalize Ellia's result for linear determinantal curves $C$ in $\PP^3$  to linear determinantal schemes $X\subset \PP^n$ of arbitrary dimension.

As in previous sections we will assume $c>1$, since the case $c=1$ corresponds to  a hypersurface $X\subset \PP^n$ and  ${\mathcal N}_X\cong \odi{X}(\delta )$, $\delta :=\deg(X)$, is $\mu$-stable. We will also assume $t>1$, since the case
$t=1$ corresponds to a codimension $c$ complete intersection $X\subset \PP^n$ and ${\mathcal N}_X=\oplus _{i=1}^{c} \odi{X}(d_i)$, $d_i\in \ZZ$, is not $\mu$-stable.

 Let us start recalling the definition of $\mu$-(semi)stability

\begin{defn} Let $X\subset \PP^n$ be a smooth projective scheme of dimension $d$ and let $\shE$ be a coherent sheaf on $X$. $\shE$ is said to be $\mu$-semistable if for any non-zero coherent subsheaf $\shF$ of $\shE$ we have the inequality $$ \mu(\shF):=\frac{deg(c_1(\shF))}{\rk(\shF)}\le \mu(\shE):=\frac{deg(c_1(\shE))}{\rk(\shE)}$$

\noindent  where as usual  $deg(c_1(\shF))=c_1(F).H^{d-1}$.
We say that $\shE$ is $\mu$-stable if strict inequality $<$ always holds.
\end{defn}

\begin{remark} \rm Recall that $\mu$-stable sheaves are simple and hence indecomposable but not vice versa.
\end{remark}

\begin{theorem} \label{semistablenormal} \footnote{Note added in proof: After
    this paper appeared online in Crelle's journal the authors were informed
    that Ph. Ellia in his paper "Double structures and normal bundles of
    spaces curves" J. London Math. Soc. 58, 18-26 (1998), Remarks and Examples
    20 (v) proved that the normal bundle of a smooth standard determinantal
    curve C in $\PP^3$ defined by a matrix with either linear or quadratic
    entries is $\mu $-semistable. Therefore, Theorem 5.3 can be seen as a
    generalization of his result in the linear case. The mentioned paper also
    contains interesting examples and results on the stability of normal
    bundles of space curves.}Let $X\subset \PP^n$ be a smooth linear
  determinantal scheme of codimension $c\ge 2$ associated to a $t\times
  (t+c-1)$ matrix $\cA$. Assume $n-c\ge 1$. Then, ${\mathcal N}_X$ is
  $\mu$-semistable.
\end{theorem}
\begin{proof} Since the notion of $\mu$-semistability is preserved when we twist by an invertible sheaf, it will be enough to prove that ${\mathcal N}_X(-H)\otimes \widetilde{S_{c-2}M}$ is $\mu$-semistable.  It follows from   Theorem \ref{normalresolution} that ${\mathcal N}_X(-H)\otimes \widetilde{S_{c-2}M}$ is a rank $c$ Ulrich sheaf on $X$    and by \cite{CH}; Theorem 2.9 any Ulrich sheaf is $\mu$-semistable which proves what we want.
\end{proof}

\begin{remark} \rm Since Ulrich sheaves are Gieseker semistable, the above proof also shows that the normal sheaf ${\mathcal N}_X$
to a smooth  linear determinantal scheme of codimension
$c\ge 2$ is Gieseker semistable.
\end{remark}

\begin{remark} \rm  Without extra hypothesis the above result cannot be improved
and the $\mu$-stability of the normal sheaf ${\mathcal N}_X$ of a linear standard determinantal
scheme $X\subset \PP ^n$  cannot be guaranteed. In fact, if we consider a rational normal curve $C\subset \PP ^n$ defined by the  $2\times 2$ minors of a $2\times n$ matrix with general linear entries, it is well known that ${\mathcal N}_{X/\PP ^n }\cong \odi{X}(2)^{n-1}.$ Therefore, ${\mathcal N}_{X}$ is $\mu$-semistable but not $\mu$-stable.
\end{remark}

In the codimension 2 case, Theorem \ref{semistablenormal} can be improved using the following lemma:

\begin{lemma} \label{Ulrichlinebundles} Let $X\subset \PP^n$ be a smooth linear determinantal scheme of codimension $c\ge 2$ and dimension $n-c\ge 2$ defined by the maximal minors of a $t\times (t+c-1)$ matrix $\cA$. Assume  $t\ge n$ when $n-c=2$. Let $H$ be a general hyperplane section of $X$ and let $Y\subset \PP^n$ be the codimension 1 subscheme of $X$ defined by the maximal minors of the $t\times (t+c)$ matrix $\cB$ obtained adding to $\cA$ a column of general linear forms.  Let $\cL$ be a line bundle on $X$. It holds:
 \begin{itemize} \item[(i)] $\cL$ is an ACM line bundle on $X$ if and only if $\cL\cong \odi{X}(aY+bH)$ with $-1\le a\le c$ and $b\in \ZZ$;
  \item[(ii)]  $\cL$ is an initialized  Ulrich
line bundle if and only if $\cL\cong \odi{X}(-Y+tH)$ or $\odi{X}(cY-cH)$.
\end{itemize}
\end{lemma}
\begin{proof} See \cite{KM2013}.
\end{proof}

In fact, we have:

\begin{theorem} \label{stablenormal} Let $X$ be a smooth linear determinantal scheme of codimension 2 in $\PP ^n$ defined by the maximal minors of a $t\times (t+1)$ matrix with linear entries. Assume that $n\ge 4$. Then, ${\mathcal
N}_X$ is $\mu$-stable. In particular, ${\mathcal
N}_X$ is simple and indecomposable.
\end{theorem}
\begin{proof} For the case $(n,t)=(4,3)$ the reader can see \cite{MP2012}; Proposition 4.10. Assume $(n,t)\ne (4,3)$. Since $\mu$-(semi)stability is preserved when we twist by line bundles, we know that ${\mathcal
N}_X(-H)$ is $\mu$-semistable (Theorem \ref{semistablenormal}) and we want to prove that it is $\mu$-stable, i.e we must rule out the existence of a coherent subsheaf $\shF\subset  {\mathcal
N}_X(-H)$ with $\mu(\shF)=\mu({\mathcal
N}_X(-H))$. By pulling-back torsion, if necessary, we may assume that ${\mathcal
N}_X(-H)/\shF$ is torsion free in which case $\shF$ is locally free and we have an exact sequence
$$ 0\longrightarrow \shF \longrightarrow {\mathcal
N}_X(-H) \longrightarrow {\mathcal
N}_X(-H)/\shF \longrightarrow 0$$
of coherent sheaves with ${\mathcal
N}_X(-H)/\shF $ torsion free and $\mu(\shF)=\mu({\mathcal
N}_X(-H)).$ By \cite{CH}; Theorem 2.9(b) $\shF$ and ${\mathcal
N}_X(-H)/\shF$ are  both Ulrich line bundles.  So, according to our Lemma \ref{Ulrichlinebundles} we have 4 possibilities:

$$ \begin{array}{lll} (1) & & 0\longrightarrow \odi{X}(-Y+tH) \longrightarrow {\mathcal
N}_X(-H) \longrightarrow \odi{X}(-Y+tH) \longrightarrow 0,\\
 (2) & &  0\longrightarrow \odi{X}(2Y-2H) \longrightarrow {\mathcal
N}_X(-H) \longrightarrow \odi{X}(-Y+tH) \longrightarrow 0, \\
(3) & &0\longrightarrow \odi{X}(-Y+tH) \longrightarrow {\mathcal
N}_X(-H) \longrightarrow \odi{X}(2Y-2H) \longrightarrow 0, \text{ or } \\
 (4) & &  \ 0\longrightarrow \odi{X}(2Y-2H) \longrightarrow {\mathcal
N}_X(-H) \longrightarrow \odi{X}(2Y-2H) \longrightarrow 0.\end{array}$$
Let us check that none of them is allowed. To this end, we will start computing the Chern classes of ${\mathcal
N}_X(-H)$. We sheafify the exact sequence (\ref{mainseq}) and we get the exact sequence:
\begin{equation}\label{exactsequencecodim2}
0\longrightarrow \odi{X}(-H)\longrightarrow \odi{X}(Y-2H)^t
\longrightarrow \odi{X}(Y-H)^{t+1}\longrightarrow {\mathcal
N}_X(-H)\longrightarrow 0.\end{equation}
Therefore, the Chern polynomial $c_u({\mathcal
N}_X(-H))$ of ${\mathcal
N}_X(-H)$ is given by
$$\begin{array}{lll} c_u({\mathcal
N}_X(-H))& = &  \sum _{u} c_i({\mathcal
N}_X(-H))u^{i} \\
& = & \frac{(1-Hu)(1+(Y-H)u)^{t+1}}{(1+(Y-2H)u)^t} \\
& = & \frac{(1-Hu)(\sum _{k=0}^{t+1} {t+1\choose k}(Y-H)^ku^k)}{
\sum _{k=0}^{t} {t\choose k}(Y-2H)^ku^k}
\end{array}$$
and a straightforward computation gives us $$c_1({\mathcal
N}_X(-H))=  Y+(t-2)H, \text{ and }$$
$$c_2({\mathcal
N}_X(-H))= -YH+\frac{t^2-t+2}{2}H^2.$$
Comparing the first Chern class we eliminate the possibility (1) and (4) because in case (1) we would get $c_1({\mathcal
N}_X(-H))=-2Y+2tH$ and in case (4) we would get $c_1({\mathcal
N}_X(-H))=4Y-4H$. To rule out the two remaining cases, we compare the second Chern class; in both cases we would get $c_2({\mathcal
N}_X(-H))=2(1+t)YH-2Y^2-2tH^2$ which is impossible and this concludes the proof of the Theorem.
\end{proof}



\begin{thebibliography}{999}

\bibitem{A} L. \ Avramov and J. \ Herzog, {\em  The Koszul algebra of a codimension 2 embedding}, Math. Z. {\bf 175}
(1980), 249--260.

\bibitem{B} L. \ Badescu,  {\em On the normal bundle of submanifolds of $\PP^n$}, Proc. Amer. Math. Soc. {\bf 136} (2008),  1505-–1513.

\bibitem{BP} B. \ Basili and C. \ Peskine, {\em D\'{e}composition du fibr\'{e} normal des surfaces lisses de $\PP^4$ et structures doubles
sur les solides de $\PP^5$}, Duke Math. J. {\bf 69} (1993), 87--95.


 \bibitem{Bu} R.O. \ Buchweitz, {\em Contributions \`{a} la theorie des singularites}, Thesis l'Universit¶e Paris VII (1981).

\bibitem{bh}
 W. \ Bruns and J. \ Herzog, {\em Cohen-Macaulay rings}, Cambridge Studies in Advanced Mathematics, {\bf 39} (1993).

\bibitem{b-v}
 W. \ Bruns and U. \ Vetter, {\em Determinantal rings},
Springer-Verlag, Lectures Notes in Mathematics {\bf 1327}, New
York/Berlin, 1988.

\bibitem{BHU}
J. \ Brennan, J. \ Herzog, and B. \ Ulrich, {em Maximally generated Cohen-Macaulay
modules}, Math. Scandinavica
{\bf 61} (1987), 181--203.

\bibitem{CH04} M. \ Casanellas and R. \ Hartshorne, {\em Gorenstein biliaison and ACM sheaves}, J.  Algebra {\bf 278} (2004), 314--341.

\bibitem{CH} M. \ Casanellas and R. \ Hartshorne, {\em Stable Ulrich bundles}, Preprint 2011, available at arXiv:1102.0878.

\bibitem{CMP} L. \ Costa, R.M. \ Mir\'o-Roig and J. \ Pons-Llopis, {\em The representation type of Segre varieties}, Adv. in Math. {\bf 230} (2012), 1995--2013.

    \bibitem{DG}
Y.~Drozd and G.M. Greuel, {\em Tame and wild projective curves and
  classification of vector bundles}, J. Algebra {\bf 246} (2001),
  1--54.


\bibitem{Ein} L. Ein, {\em On the Cohomology of projective Cohen-Macaulay
determinantal varieties of $\PP^n$}, in Geometry of complex projective
varieties (Cetraro, 1990), 143-152, Seminars and Conferences 9,
Mediterranean Press (1993).

\bibitem{eise}  D. \ Eisenbud, {\em Commutative Algebra. With a view toward
algebraic geometry}, Springer-Verlag, Graduate Texts in
Mathematics {\bf 150 } (1995).



\bibitem{ESW} D. Eisenbud, F. Schreyer  and J. Weyman,  {\em Resultants and chow forms via exterior syzygies},
J. of {A}mer. {M}ath. {S}oc.,  {\bf 16} (2003), 537--579.

\bibitem{ES} D. \ Eisenbud and F. \ Schreyer, {\em Boij-S\"{o}derberg Theory}, Combinatorial aspects
of commutative algebra and algebraic geometry, Proceeding of the Abel
Symposium, 2009.


\bibitem{El}  Ph. \ Ellia, {\em Exemples de courbes de $\PP^3$ \`{a} fibr\'{e} normal semi-stable, stable}
Math. Ann. {\bf 264} (1983), no. 3, 389--396.

\bibitem{EH} G. \ Ellingsrud and A. \ Hirschowitz, {\em Sur le fibr\'{e} normal des courbes gauches}, C. R. Acad. Sci. Paris S\'{e}r. I Math. {\bf 299} (1984), 245--248.

\bibitem{FF} D. \ Faenzi and M.L. \ Fania, {\em On the Hilbert scheme of
    determinantal subvarities}, Math. Res. Lett. {\bf 21} no. 2, 297-311.

\bibitem{Hor} G. Horrocks,  {\em Vector bundles on the punctual spectrum of a local ring},
    Proc. London Math. Soc., {\bf 14}  (1964), 689--713.

\bibitem{Mac2} D. \ Grayson and M. \ Stillman,  \newblock Macaulay 2-a
  software system for algebraic geometry and commutative algebra,
  available at http://www.math.uiuc.edu/Macaulay2/ .


\bibitem{SGA2}
A. \ Grothendieck, {\em Cohomologie locale des faisceaux
coh\'{e}rents et Th\'{e}or\`{e}mes de Lefschetz locaux et
globaux}, North Holland, Amsterdam (1968).

\bibitem{Lo} A. L\'opez, {\em Noether-Lefschetz Theory and the Picard group of
projective surfaces}, Memoirs A.M.S. \textbf{438}
(1991).

\bibitem{Ki} D. \ Kirby, {\em A sequence of complexes associated to a matrix}, J. London Math. Soc. {\bf 7} (1973), 523--530.

  \bibitem{KP} J.O. \ Kleppe and C. \ Peterson, {\em Maximal
Cohen-Macaulay modules and Gorenstein algebras}, J. Algebra {\bf 238} (2001) 776-800.

 \bibitem{KP2} J.O. \ Kleppe and C. \ Peterson, {\em Sheaves with Canonical Determinant on Cohen-Macaulay Schemes}, J. Algebra {\bf 256} (2002) 250-279.

\bibitem{K2007} J.O. \ Kleppe,   {\em Unobstructedness and dimension of families of Gorenstein algebras},
Collect. Math. {\bf 58} (2007), 199-238.


\bibitem{K2011} J.O. \ Kleppe, {\em Deformations of modules of maximal grade and
      the Hilbert scheme at determinantal schemes.}, J. Algebra {\bf 407}, (2014), 246-274


\bibitem{KM05} J.O. \ Kleppe and R.M. \ Mir\'o-Roig, {\em Dimension of
families of determinantal schemes},  Trans. Amer.  Math. Soc. {\bf
357}, (2005), 2871-2907.

\bibitem{KM11} J.O. \ Kleppe and R.M. \ Mir\'o-Roig, {\em Families of Determinantal Schemes},  Proc. Amer.  Math. Soc. {\bf
139}, (2011), 3831-3843.

\bibitem{KM2013} J.O. \ Kleppe and R.M. \ Mir\'o-Roig, {\em  In preparation}.

\bibitem{MP2012}
R.M. \ Mir\'o-Roig and J. \ Pons-Llopis,
{\em   Representation type of rational ACM surfaces $X\subseteq\PP^4$},
 Algebras and Representation Theory, {\bf 16} (2013) 1135-1157.


\bibitem{MP}
R.M. \ Mir\'o-Roig and J. \ Pons-Llopis, {\em  $N$-dimensional Fano varieties of wild representation type}, Jour. Pure and Applied Alg., to appear. Preprint, available from arXiv:1011.3704.

  \bibitem{MR} R.M. \ Mir\'o-Roig, {\em The representation type of rational normal scrolls}, Rendiconti
                del Circolo di Palermo {\bf 62} (2013) 153-164.

\bibitem{P} D. \ Perrin, {\em Courbes passant par m points g\'{e}n\'{e}raux de $\PP^3$}, M\'{e}m. Soc. Math. France {\bf 28-29} (1987).

\bibitem{S} G. \ Sacchiero, {\em Exemple de courbes de $\PP ^3$ de fibr\'{e} normal stable}, Comm. Algebra {\bf 11} (1983),  2115–-2121.

\bibitem{Ulr} B. Ulrich, {\em Gorenstein rings and modules with high number of generators},
    Math. Z., {\bf 188} (1984),23--32.

\end{thebibliography}
\end{document}